\documentclass[11pt,reqno]{amsart}
\usepackage{amsmath, amssymb, amsfonts, xcolor}
\usepackage{hyperref}
\numberwithin{equation}{section}
\usepackage{amsthm}
\newtheorem{thm}{Theorem}[section]
\newtheorem{example}{Example}
\newtheorem{prop}[thm]{Proposition}
\newtheorem{lem}[thm]{Lemma}
\newtheorem{conj}[thm]{Conjecture}
\newtheorem{dfn}[thm]{Definition}

\newtheorem{remark}[thm]{Remark}
\newtheorem{cor}[thm]{Corollary}

\usepackage[a4paper, top = 1.2in, bottom = 1.2in, left = 1.2in, right = 1.2in]{geometry}
\usepackage{enumitem}
\newlist{steps}{enumerate}{1}
\setlist[steps, 1]{label = Step \arabic*:}

\numberwithin{equation}{section}

\newcommand{\C}{\mathbb{C}}
\newcommand{\F}{\mathbb{F}}
\newcommand{\N}{\mathbb{N}}
\newcommand{\Q}{\mathbb{Q}}

\newcommand{\Z}{\mathbb{Z}}

\newcommand{\mcO}{\mathcal{O}}

\newcommand{\mfc}{\mathfrak{c}}

\newcommand{\mfm}{\mathfrak{m}}
\newcommand{\mfn}{\mathfrak{n}}

\newcommand{\mfp}{\mathfrak{p}}
\newcommand{\mfq}{\mathfrak{q}}
\newcommand{\mfP}{\mathfrak{P}}
\newcommand{\mfQ}{\mathfrak{Q}}

\newcommand{\GL}{\mathrm{GL}}

\newcommand{\Cl}{\mathrm{Cl}}

\newcommand{\Gal}{\mathrm{Gal}}

\def\1{1\!\!1}
\newcommand{\mrm}[1]{\mathrm{#1}}

\title[Asymptotic Fermat equation of signature $(r, r, p)$ over $K$]{Asymptotic Fermat equation of signature $(r, r, p)$ over totally real fields}

\author[S. Jha]{Somnath Jha}
\address[S. Jha]{Department of Mathematics and Statistics, IIT Kanpur, Kanpur 208016, India}
\email{jhasom@gmail.com}

\author[S. Sahoo]{Satyabrat Sahoo}
\address[S. Sahoo]{Yau Mathematical Sciences Center, Tsinghua University, Beijing 100084, China
	\& \newline
	Department of Mathematics and Statistics, IIT Kanpur, Kanpur 208016, India}
\email{sahoos@mail.tsinghua.edu.cn}

\date{\today}

\begin{document}
	%\maketitle
	\begin{abstract}
		Let $K$ be a totally real number field and $ \mcO_K$ be the ring of integers of $K$. This manuscript examines the asymptotic solutions of the Fermat equation of signature $(r, r, p)$, specifically $x^r+y^r=dz^p$ over $K$,
		%      In this article, we study the asymptotic solutions of the Fermat equation of signature $(r, r, p)$, i.e., $x^r+y^r=dz^p$ over $K$,
		where $r,p \geq5$ are rational primes and odd $d\in \mcO_K \setminus \{0\}$. For a certain class of fields $K$, we first prove that the equation $x^r+y^r=dz^p$ has no asymptotic solution $(a,b,c) \in \mcO_K^3$ with $2 |c$. We also study the asymptotic solutions $(a,b,c) \in \mcO_K^3$ to the equation $x^5+y^5=dz^p$ when $2 \nmid c$. We use the modular method to prove these results. 
		%      Finally, we provide some local criteria of $K$ for these results.
	\end{abstract}
	
	\maketitle
	
	%    \tableofcontents
	\section{Introduction}
	After Wiles' groundbreaking proof of Fermat's Last Theorem, significant progress has been made in studying the generalized Fermat equation, i.e.,  
	\begin{equation}
		\label{generalized Fermat eqn}
		Ax^p+By^q=Cz^r, \text{ where } A,B,C, p,q,r \in \Z \setminus \{0\},
	\end{equation}
	$\gcd(A,B,C)=1$ and $ p,q,r \geq 2$ with $\frac{1}{p} +\frac{1}{q}+ \frac{1}{r} <1$. We say $(p,q,r)$ as the signature of \eqref{generalized Fermat eqn}.
	The generalized Fermat equation~\eqref{generalized Fermat eqn} is subject to the following conjecture. 
	\begin{conj}
		\label{DG conj}
		Fix  $ A,B,C \in \Z $ with $\gcd(A,B,C)=1$. Then, over all choices of prime exponents $ p,q,r$ with $\frac{1}{p} +\frac{1}{q}+ \frac{1}{r} <1$, the equation~\eqref{generalized Fermat eqn} has only finitely many non-trivial primitive integer solutions (here the solutions like $1^p+2^3=3^2$ counted only once for all $p$).
		%integers $a,b,c \in \Z$ with $\gcd(a,b,c)=1$ and such that $Ax^p+By^q=Cz^r$.
	\end{conj}
	Recall that a solution $(a,b,c) \in \Z^3$ of the equation \eqref{generalized Fermat eqn} is said to be non-trivial if $abc \neq 0$ and primitive if $\gcd(a,b,c)=1$.
	In \cite{DG95}, Darmon and Granville proved that for fixed integers $ A,B,C \in \Z \setminus \{0\}$ and for fixed primes $ p,q,r$ with $\frac{1}{p} +\frac{1}{q}+ \frac{1}{r} <1$, the equation~\eqref{generalized Fermat eqn} has only finitely many non-trivial primitive integer solutions.
	
	Throughout, we fix a totally real number field. We denote it by $K$, and by $\mathcal{O}_K$ its ring of integers. In this article, we study the solutions of the Fermat equation of signature $(r,r,p)$, i.e.,
	\begin{equation}
		\label{r,r,p over Z}
		x^r+y^r=dz^p,
	\end{equation}
	over $K$, where $r,p\geq 5$ are rational primes and $d \in \mcO_K \setminus \{0\}$.
	%   In \cite{K98}, Kraus first studied integer solutions of the equation~\ref{r,r,p over Z} with $r=3$ and $d=1$.
	
	At first, we review the known results of the Diophantine equation~\eqref{r,r,p over Z} with $r=5$. The integer solutions of the equation~\eqref{r,r,p over Z} were first studied by Billerey in~\cite{B07} for $d=2^\alpha.3^\beta. 5^\gamma$ with $0 \leq \alpha, \beta, \gamma \leq 4$. In~\cite{BD10}, Billerey and  Dieulefait first generalizes the results of \cite{B07} to $d=2^\alpha.3^\beta. 5^\gamma$ with $\alpha \geq 2$ and $\beta, \gamma$ arbitary. Then, in \cite{BD10}, they also proved that the equation~\eqref{r,r,p over Z} with $d=7$ (respectively $d=13$) has no non-trivial primitive integer solutions for $p \geq 13$ (respectively $p \geq 19$). Further, in \cite{DF14}, Dieulefait and Freitas proved that the equation~\eqref{r,r,p over Z} with $d=2,3$ has no non-trivial primitive integer solutions for a set of primes of density $\frac{3}{4}$. Finally, in \cite{BCDF19}, Billerey, Chen, Dieulefait and Freitas proved that the equation~\eqref{r,r,p over Z} with $d=3$ has no non-trivial primitive integer solutions for $p \geq 2$.
	
	In \cite{F15}, Freitas constructed a multi-Frey family of curves to study the solutions of the equation~\eqref{r,r,p over Z} and proved that the equation~\eqref{r,r,p over Z} (with $r=7$, $d=3$) has no non-trivial primitive integer solutions for $p> (1+3^{18})^2$. In \cite{BCDF23b}, the authors proved that the equation $x^7+y^7=3z^n$ (respectively $x^7+y^7=z^n$) with integers $n \geq 2$ has non-trivial primitive integer solution $(a,b,c)$ (respectively $(a,b,c)$ of the form $2 |a+b$ or $7 |a+b$).  
	%     In \cite{BCDF23b}, the authors also studied the integer solutions of $x^7+y^7=z^n$ and proved that $x^7+y^7=z^n$ has no non-trivial primitive solution $(a,b,c) \in \Z^3$ of the form $2 |a+b$ or $7 |a+b$. 
	In \cite{BCDF23a}, the authors studied the integer solutions of the equation $x^{11}+y^{11}=z^n$ with integers $n \geq 2$ using Frey abelian varieties. The integer solutions of the equation \eqref{r,r,p over Z} for $r=13$ were first studied in \cite{DF13}. In \cite{BCDF19} (respectively \cite{BCDDF23}), the authors proved the equation \eqref{r,r,p over Z} with $r=13$, $d=3$, and $p \neq 7$ (respectively for all $p$) has no non-trivial primitive integer solution.  
	%  In \cite{BCDDF23}, the authors proved that $x^{13}+y^{13}=3z^n$ with integers $n \geq 2$ has no non-trivial primitive integer solution. 
	
	In \cite{FN24}, the authors proved that for a set of primes $p$ with positive density and for fixed $r, d$ with $r,p \nmid d$, the equation~\eqref{r,r,p over Z} has no non-trivial primitive integer solution $(a,b,c)$ of the form $2 |a+b$ or $r |a+b$. In a recent preprint \cite{KMO24}, it is studied the asymptotic integer solutions of the equation \eqref{r,r,p over Z} by assuming the weak Frey-Mazur conjecture  (cf. \cite[Conjecture 3.11]{KMO24}) and the Eichler–Shimura conjecture  (cf. \cite[Conjecture 3.9]{KMO24}).  
	
	In \cite{M23}, Mocanu studied the asymptotic solution of a certain type to the equation $x^r+y^r=z^p$ over $K$ using the modular method. A comprehensive survey on the solutions of the Diophantine equation\eqref{generalized Fermat eqn} over $K$ can be found in \cite{KS24 survey article}. In general, the literature for the generalized Fermat equation of signature $(p,p,p)$, $(p,p,2)$, and $(p,p,3)$ over $K$ using the modular method is given by the following table. 
	%Let $r,n \in \N$. 
	%In the following table, we discuss various results of the generalized Fermat equation of signature $(p,p,p)$, $(p,p,2)$, $(p,p,3)$ over $K$. More precisely:
	%\clearpage
	\begin{table}[hbt!]
		\centering
		\begin{tabular}{|c|c|c|}  
			\hline  
			Fermat-type equation & Solutions over $\Q$ & Asymptotic solutions over $K$ \\ \hline
			$x^p+y^p=z^p$ & \cite{W95} & \cite{JM04}, ~\cite{FS15}, \cite{FS15b} \\ \hline
			$Ax^p+By^p=Cz^p$ & \cite{K97}, \cite{DM97}, \cite{R97} & \cite{D16}, \cite{KS23 Diophantine1}, \cite{S24 GFLT}\\ \hline
			$x^p+y^p=z^2$ & \cite{DM97} &~\cite{IKO20}, \cite{M22}, \cite{KS23 Diophantine1}, \cite{KS23 Diophantine2} \\ \hline
			%$x^p+2^ry^p=z^2$& \cite{I03}, \cite{S03} & \cite{KS23 Diophantine2}\\ \hline 
			%$x^p+2^ry^p=2z^2$& \cite{I03}& \cite{KS23 Diophantine2}\\ \hline  
			$Ax^p+By^p=Cz^2$& \cite{I03}, \cite{S03}, \cite{BS04}, \cite{C24} & \cite{KS23 Diophantine2} for $C=1$, $2$\\ \hline
			$x^p+y^p=z^3$ & \cite{DM97} & \cite{M22}, \cite {IKO23}, \cite{KS24 GAFLT}\\ \hline  
			$Ax^p+By^p=Cz^3$& \cite{BVY04} & \cite{KS24 GAFLT}\\ \hline
		\end{tabular}  
		% \caption{Solutions of Fermat-type equation of signature $(n,n,2)$.}
		% *\label{table 1}
	\end{table}

	In this manuscript, we study the asymptotic solution of the Diophantine equation $x^r+y^r=dz^p$ over
	%	 totally real number fields
	$K$, where $r,p\geq 5$ rational primes and $d\in \mcO_K \setminus \{0\}$. Let $\zeta_r$ be a primitive $r$th root of unity in $\C$ and  $K^+:= K(\zeta_r+ \zeta_r^{-1})$. The main results of this article are Theorems~\ref{main result1 for (r,r,p)},~\ref{main result2 for (r,r,p)}.
	\begin{itemize}
		\item In Theorem~\ref{main result1 for (r,r,p)}, we prove that for a certain class of totally real fields $K$, there exists a constant $V>0$ (depending on $K,r,d$) such that for primes $p>V$, the equation $x^r+y^r=dz^p$ has no non-trivial primitive solution $(a,b,c) \in \mcO_{K^+}^3$ with $\mfP |c$ for some prime ideal $\mfP$ of $K^+$ lying above $2$.
		%(cf. \S\ref{section for main result of (r,r,p)} for the definition of $K^+$ and $S_{{K^+}, 2}$). 
		In particular, in Corollary~\ref{cor to main result1} (respectively Corollary~\ref{cor2 to main result1}), we show that the equation $x^r+y^r=dz^p$ with $p>V$ has no non-trivial  primitive solution $(a,b,c) \in \mcO_K^3$ with $2 |c$ (respectively $2 |a+b$).
		
		\item In Theorem~\ref{main result2 for (r,r,p)}, we study the asymptotic solution of the equation $x^5+y^5=dz^p$ over $K$ and show that there exists a constant $V>0$ (depending on $K,d$) such that for primes $p>V$, the equation $x^5+y^5=dz^p$ has no solution in $W_K$ (cf. Definition~\ref{defn for W_K} for $W_K$).
		% non-trivial primitive solution $(a,b,c) \in \mcO_{K}^3$ with $p \nmid a+b$, $\mfP \nmid c$, and $c \neq \pm1$ for some prime $\mfP$ in $K$ lying above $2$. 
		In particular, in Corollary~\ref{cor to main result2}, we show that the equation $x^5+y^5=dz^p$ with $p>V$ has no non-trivial primitive solution $(a,b,c) \in \mcO_{K}^3$ with $c\neq \pm1$ and $2 \nmid c$.
	\end{itemize}
	To prove Theorems~\ref{main result1 for (r,r,p)},~\ref{main result2 for (r,r,p)}, we use the modular method inspired by Freitas and Siksek in \cite{FS15}. 
	We recall the broad general strategy and the key steps in the modular method:  
	\begin{steps}
		\item Given any non-trivial primitive solution $(a, b, c)\in \mcO_{K}^3$ to a Fermat-type equation of exponent $p$, associate a suitable Frey elliptic curve $E/K$.
		\item Then we prove the modularity and semistability of $E/K$ for $p \gg 0$, and the mod $p$ Galois representation $\bar{\rho}_{E,p}$ is irreducible for $p \gg 0$. Using level lowering results,  find $\bar{\rho}_{E,p} \sim \bar{\rho}_{f,p}$, for some Hilbert modular newform $f$ over $K$ of parallel weight $2$ with rational eigenvalues, so that the level $N_p$ of $f$ strictly divides the conductor of $E$.
		\item Prove that the finitely many Hilbert modular newforms that occur in the above step do not correspond to $\bar{\rho}_{E,p}$, to get a contradiction.
	\end{steps}
	{\bf Methodology:}  We now discuss some key steps in the proof of our main results, i.e., Theorem~\ref{main result1 for (r,r,p)} and Theorem~\ref{main result2 for (r,r,p)}. 
	\begin{itemize}
		\item 	The Frey curves corresponding to the equations $x^r+y^r=dz^p$ and $x^5+y^5=dz^p$ are attached by Freitas in \cite{F15} over $K^+$ and Billerey in \cite{B07} over $K$, respectively (cf. \eqref{Frey curve result1}, \eqref{Frey curve result2} for the Frey curves).
		
		\item We prove modularity and semistability of the Frey curve $E$ in \eqref{Frey curve result1} over $K^+$ (respectively $E$ in \eqref{Frey curve result2} over $K$) with $p \gg0$.
		
		\item  Using the results of \cite{FS15}, we get $\bar{\rho}_{E,p} \sim \bar{\rho}_{f,p}$, for some Hilbert modular newform $f$ of lower level, as specified in Step $2$ above.
		
		\item To get a contradiction, as outlined in Step $3$, Theorem~\ref{main result1 for (r,r,p)} uses technical ideas from \cite{FS15}, \cite{M23} which involves the $S_{K^+,2rd}$-unit equation \eqref{S_K-unit solution}, while Theorem~\ref{main result2 for (r,r,p)} uses technical ideas from \cite{M22} which involves the equation~\eqref{eqn for main result2}.
		
		\item More precisely, in Step $3$, for any prime $\mfP \in S_{K^+,2}$ (respectively $\mfP \in S_{K,2}$) and for the Frey curve $E/K^+$ in \eqref{Frey curve result1} (respectively $E/K$ in \eqref{Frey curve result2}) with $p \gg 0$, we show that $v_\mfP(j_E) < 0$ and $p \nmid v_\mfP(j_E)$ (cf. Lemmas~\ref{reduction on T and S},~\ref{reduction on T and S main result2}). Using a level-lowering result of~\cite{FS15}, we show that there exists a constant $V>0$ (depending on $K,r,d$) and an elliptic curve $E'/K^+$ having full $2$-torsion points (respectively $E'/K$ having a non-trivial $2$-torsion point) such that $\bar{\rho}_{E,p} \sim \bar{\rho}_{E^\prime,p}$ for $p>V$ (cf. Propositions~\ref{auxilary result for main result1},~\ref{auxilary result for main result2}). Then, using Lemma~\ref{criteria for potentially multiplicative reduction}, we get $v_\mfP(j_{E^\prime})<0$. 
		Finally, we relate $j_{E'}$ in terms of the solution of the $S_{K^+, 2rd}$-unit equation~\eqref{S_K-unit solution} (respectively equation~\eqref{eqn for main result2}), together with the explicit bound \eqref{assumption for main result1} (respectively \eqref{assumption for main result2}) to get $v_\mfP(j_{E^\prime}) \geq0$ for some $\mfP \in S_{K^+,2}$ (respectively $\mfP \in S_{K,2}$), which is a contradiction.
		In particular, we stress that the methods used to prove Theorem~\ref{main result1 for (r,r,p)} and Theorem~\ref{main result2 for (r,r,p)} are different.
	\end{itemize}

	%	To arrive at a contradiction in Step $3$ for Theorem~\ref{main result1 for (r,r,p)}, we use an explicit bound on the solutions of $S_{K^+,2rd}$-unit equation \eqref{S_K-unit solution}. On the other hand, for proving Theorem~\ref{main result2 for (r,r,p)}, we use an explicit bound on the solutions of the equation~\eqref{eqn for main result2}, to get a contradiction in Step $3$. 
	
	%	\medskip

	\subsection{Structure of the article}
	In \S\ref{notations section for x^r+y^r=dz^p}, we state the main results of this article, i.e., Theorems~\ref{main result1 for (r,r,p)},~\ref{main result2 for (r,r,p)}. In \S\ref{steps to prove main result1} (respectively \S\ref{steps to prove main result2}) we prove Theorem~\ref{main result1 for (r,r,p)} (respectively  Theorem~\ref{main result2 for (r,r,p)}). In \S\ref{section for local criteria}, we give various local criteria on $K$ such that Theorem~\ref{main result1 for (r,r,p)} holds over $K^+$.

	\section{Statement of the main results}
	\label{notations section for x^r+y^r=dz^p}
	Let $K$ denote a totally real number field, and $\mcO_K$ denote the ring of integers of $K$. In this section, we study the solutions of the Diophantine equation
	\begin{equation}
		\label{r,r,p over K}
		x^r+y^r=dz^p
	\end{equation} 
	over $K$, where $r,p\geq 5$ rational primes, and $d\in \mcO_K \setminus \{0\}$.
	\begin{dfn}
		A solution $(a, b, c)\in \mcO_K^3$ to the equation \eqref{r,r,p over K} is said to be trivial if $abc=0$,
		otherwise it is non-trivial.
		We say $(a, b, c)\in \mcO_K^3$ is primitive if $a\mcO_K+b\mcO_K+c\mcO_K=\mcO_K$.
	\end{dfn}  
	%\begin{dfn}
	%	\label{def for W_K}
	%	%			\label{remark for W_K}
	%	Let $W_{K^+}$ be the set of all non-trivial primitive solutions $(a, b, c)\in  \mcO_{K^+}^3$ to the equation~\eqref{Ax^p+By^p+Cz^p=0} with $\mfP |abc$ for all $\mfP \in S_K$.
	%\end{dfn}

	%\begin{dfn}
	%	\label{asymptotic solution}
	%	Let $F$ be a number field. We say the equation $x^r+y^r=Cz^p$ of prime exponent $p$ has no asymptotic solution in a set $S \subseteq \mcO_F^3$, if there exists a constant $V_{F}>0$ (depending on $F$) such that for primes $p>V_{F}$, the equation $x^r+y^r=Cz^p$ has no non-trivial primitive solution in $S$.
	%\end{dfn}
	
	%\begin{remark}
	%	\label{remark for W_K}
	%	Let $\mfP \in S_K$. If $(a, b, c)\in W_K$ is a solution of \eqref{Ax^p+By^p+Cz^p=0} of exponent $p > \max\{v_\mfP (A), v_\mfP (B), v_\mfP (C) \}$, then $\mfP$ divides exactly one of $a$, $b$ and $c$. Otherwise, let $\mfP$ divides both $a$ and $b$. Then $\mfP^p | Aa^p+Bb^p=-Cc^p$. Since $p > v_\mfP (C)$, $\mfP |c$, which is not possible since $(a, b, c)$ is primitive. Similarly the other cases, i.e.,  $\mfP$ divides both $b$ and $c$, and $\mfP$ divides both $a$ and $c$, are not possible.
	%\end{remark}

	\subsection{Main result for the solutions of $x^r+y^r=dz^p$ over $K$}
	\label{section for main result of (r,r,p)}
	Throughout this subsection, we assume $d$ is odd (in the sense that $\mfP \nmid d$ for all prime ideals $\mfP$ of $K$ lying above $2$). For any number field $F$, let $P_F$ denote the set of all prime ideals of $\mcO_F$. 
	For any set $S \subseteq P_F$, let $\mcO_{S}:=\{\alpha \in F : v_\mfq(\alpha)\geq 0 \text{ for all } \mfq \in P_F \setminus S\}$ be the ring of $S$-integers in $F$ and $\mcO_{S}^*:=\{\alpha \in F : v_\mfq(\alpha)= 0 \text{ for all } \mfq \in P_F \setminus S\}$ be the $S$-units of $\mcO_{S}$.   
	
	For any number field $F$ and any element $\alpha \in \mcO_F \setminus \{0\}$, let 
	$$S_{F,\alpha}:=\{\mfp: \mfp\in P_F \text{ with } \mfp |\alpha\}.$$
	Let $L := K(\zeta_r)$. Let $K^+:=K(\zeta_r+ \zeta_r^{-1})$ be the maximal totally real subfield of $L$. 
	%Let $S_{K^+}:= S_{K^+, 2rd}$.
	We now state the first main result of this article.
	%We now show that equation~\eqref{r,r,p over K} has no asymptotic solution $(a,b,c) \in \mcO_{K^+}^3$ with $\mfP |c$, for some $\mfP \in S_{{K^+}, 2}$. More precisely:
	\begin{thm}[Main result 1]
		\label{main result1 for (r,r,p)}
		Let $K$ be a totally real number field, $r\geq 5$ be a fixed rational prime and $d\in \mcO_K \setminus \{0\}$ be odd. Let $K^+:=K(\zeta_r+ \zeta_r^{-1})$. Suppose, for every solution $(\lambda, \mu)$ to the $S_{K^+, 2rd}$-unit equation
		\begin{equation}
			\label{S_K-unit solution}
			\lambda+\mu=1, \ \lambda, \mu \in \mcO_{S_{K^+, 2rd}}^\ast,
		\end{equation}
		there exists a prime $\mfP \in S_{{K^+}, 2}$ that satisfies
		\begin{equation}
			\label{assumption for main result1}
			\max \left\{|v_\mfP(\lambda)|,|v_\mfP(\mu)| \right\}\leq 4v_\mfP(2).
		\end{equation}
		Then, there exists a constant $V=V_{K,r,d}>0$ (depending on $K,r,d$) such that for primes $p>V$, the equation $x^r+y^r=dz^p$ has no non-trivial primitive solution $(a,b,c) \in \mcO_{K^+}^3$ with $\mfP |c$.
	\end{thm}
	\begin{remark}
		For any number field $F$ and any finite set $S \subseteq P_F$, the $S$-unit equation $\lambda+\mu=1, \ \lambda, \mu \in \mcO_{S}^\ast$ has only a finite number of solutions (cf. \cite{S14} for more details). Consequently, the $S_{K^+,2rd}$-unit equation~\eqref{S_K-unit solution} also has only a finite number of solutions in $K^+$. Moreover, the solutions of the $S$-unit equation are effectively computable (cf. \cite{AKMRVW21}).
		Thus, in Theorem~\ref{main result1 for (r,r,p)}, we reduce a Diophantine problem to a computable problem. A similar bound on $S_{K^+,2rd}$-unit equation~\eqref{S_K-unit solution} can also be found in \cite[Theorem 3]{FS15} (respectively \cite[Theorem 4]{M23}) but they have considered  $S_{K,2}$-unit equation (respectively $S_{K^+,2r}$-unit equation) instead of $S_{K^+,2rd}$-unit equation.
	\end{remark}
	In \S\ref{section for local criteria}, we give several classes of fields $K$ for which Theorem~\ref{main result1 for (r,r,p)} holds over $K^+$. As an immediate corollary to Theorem~\ref{main result1 for (r,r,p)}, we study the solutions of the equation~\eqref{r,r,p over K} over $K$.
	\begin{cor}
		\label{cor to main result1}
		Let $K, \mfP, r,d$ satisfy the hypothesis of Theorem~\ref{main result1 for (r,r,p)}. Let $\mfQ$ be the unique prime of $K$ lying below $\mfP$. Then, there exists a constant $V=V_{K,r,d}>0$ (depending on $K,r,d$) such that for primes $p>V$, the equation $x^r+y^r=dz^p$ has no non-trivial primitive solution $(a,b,c) \in \mcO_K^3$ with $\mfQ |c$. In particular, the equation $x^r+y^r=dz^p$ with $p>V$ has no non-trivial primitive solution $(a,b,c) \in \mcO_K^3$ with $2 |c$.
	\end{cor}
	
	The following corollary describes the asymptotic solution of the equation~\eqref{r,r,p over K} with $2 |a+b$. 
	%More precisely:
	\begin{cor}
		\label{cor2 to main result1}
		Let $K, \mfP,r, d$ satisfy the hypothesis of Theorem~\ref{main result1 for (r,r,p)}. Then, there exists a constant $V=V_{K,r,d}>0$ (depending on $K,r,d$) such that for primes $p>V$, the equation $x^r+y^r=dz^p$ has no non-trivial primitive solution $(a,b,c) \in \mcO_K^3$ with $2 |a+b$.
	\end{cor}
	
	\begin{proof}
		%We will prove this corollary using the contradiction method. 
		Suppose $(a,b,c) \in \mcO_K^3$ is a non-trivial primitive solution to~\eqref{r,r,p over K} of exponent $p>V$ with $2 |a+b$, where $V$ is the same constant as in Corollary~\ref{cor to main result1}. This gives $dc^p=a^r+b^r=  (a+b) \times \sum_{i=0}^{r-1}a^{r-1-i}b^i$. Since $2 |a+b$ and $2 \nmid d$, it follows that $2 |c$. This contradicts Corollary~\ref{cor to main result1}. %Hence the proof follows.
	\end{proof}
	
	\subsection{Main result for the solutions of $x^5+y^5=dz^p$ over $K$}
	In this subsection, we study the solutions of the Diophantine equation
	\begin{equation}
		\label{5,5,p}
		x^5+y^5=dz^p
	\end{equation} 
	over $K$, where $p\geq 5$ is a rational prime and $d \in \mcO_K\setminus \{0\}$ is $5$-th power free (in the sense that there doesn't exist any prime $\mfp \in P_K$ such that $\mfp^5 |d$). Throughout this subsection, we assume $d$ is even (in the sense that $\mfP |d$ for some prime  $\mfP \in S_{K,2}$).
	\begin{remark}
		\label{coprime of a,b,c}
		If $(a, b, c)\in \mcO_K^3$ is a non-trivial primitive solution of the equation~\eqref{5,5,p}, then $a,b,c$ are pairwise coprime since $d$ is $5$-th power free.
	\end{remark}
	\begin{dfn}
		\label{defn for W_K}
		Let $W_K$ be the set of all non-trivial primitive solutions $(a,b,c) \in \mcO_{K}^3$ to the equation~\eqref{5,5,p} such that $c\neq \pm1$ and $\mfP \nmid c$ for all $\mfP \in S_{K,2}$.
	\end{dfn}
	For any set $S \subseteq P_K$ and $m \in \N$, let $\Cl_S(K):= \mrm{Cl}(K)/\langle [\mfP]\rangle_{\mfP\in S}$ and $\Cl_S(K)[m]$ be the $m$-torsion points of $\Cl_S(K)$. We now state the second main result of this article.
	%Let $S_{K}:= S_{K, 2}$. 
	\begin{thm}[Main result 2]
		\label{main result2 for (r,r,p)}
		\label{main result2 for r,r,p}
		Let $K$ be a totally real number field and $d\in \mcO_K \setminus \{0\}$. Let $\Cl_{S_{K,10d}}(K)[2]=\{1\}$. Suppose for every solution $(\alpha, \beta, \gamma) \in \mcO_{S_{K,10d}}^\ast \times \mcO_{S_{K,10d}}^\ast \times \mcO_{S_{K,10d}}$ to the equation
		\begin{equation}
			\label{eqn for main result2} 
			\alpha+\beta=\gamma^2,
		\end{equation}
		there exists a prime $\mfP \in S_{K, 2}$ that satisfies 
		\begin{equation}
			\label{assumption for main result2}
			\left| v_\mfP \left(\alpha \beta^{-1}\right) \right| \leq 6v_\mfP(2).
		\end{equation}
		If $v_\mfP(d)> 4v_\mfP(2)$, then there exists a constant $V=V_{K,d}>0$ (depending on $K,d$) such that for primes $p>V$, the equation $x^5+y^5=dz^p$ has no solution in $W_K$.
		%non-trivial primitive solution $(a,b,c) \in \mcO_{K}^3$ with $p \nmid a+b$, $c\neq \pm1$ and $\mfP \nmid c$.
		%\begin{enumerate}
		%	\item If $d \in \mcO_K\setminus \{0\}$, then there exists a constant $V=V_{K,r,d}>0$ (depending on $K,r,d$) such that for primes $p>V$, the equation $x^5+y^5=dz^p$ has no non-trivial primitive solution $(a,b,c) \in \mcO_{K}^3$ with $c\neq\pm1$ and $p \nmid a+b$.
		%	\item If $d=\pm u 2^r 5^s$, where $u \in \mcO_K^\ast,\ r,s \in \Z_{\geq 0}$, then there exists a constant $V=V_{K,r,d}>0$ (depending on $K,r,d$) such that for primes $p>V$, the equation $x^5+y^5=dz^p$ has no non-trivial primitive solution $(a,b,c) \in \mcO_{K}^3$ with $c\neq\pm1$
		%\end{enumerate}
		%	 For $d \in \mcO_K\setminus \{0\}$ (respectively $d=\pm u 2^r 5^s$, where $u \in \mcO_K^\ast,\ r,s \in \Z_{\geq 0}$) with $v_\mfP(d)> 4v_\mfP(2)$, there exists a constant $V=V_{K,r,d}>0$ (depending on $K,r,d$) such that for primes $p>V$, the equation $x^5+y^5=dz^p$ has no non-trivial primitive solution $(a,b,c) \in \mcO_{K}^3$ with $c\neq\pm1$ and $p \nmid a+b$.
	\end{thm}
	%As an immediate corollary to the above theorem, we have the following result.
	%\begin{cor}
	%	\label{cor to main result2}
	%	Let $K, \mfP$ satisfy the hypothesis of Theorem~\ref{main result2 for (r,r,p)}. Assume $2$ is inert in $K$. Then, there exists a constant $V=V_{K,r,d}>0$ (depending on $K,r,d$) such that for primes $p>V$, the equation $x^r+y^r=dz^p$ has no non-trivial primitive solution $(a,b,c) \in \mcO_K^3$ with $\mfQ |c$. In particular, the equation $x^r+y^r=dz^p$ with $p>V$ has no non-trivial primitive solution $(a,b,c) \in \mcO_K^3$ with $2 |c$.
	%\end{cor}
	
	\begin{remark}
		For any finite set $S \subseteq P_K$ and any two solutions $(\alpha, \beta, \gamma), (\alpha', \beta', \gamma') \in \mcO_{S}^\ast \times \mcO_{S}^\ast \times \mcO_{S}$ to the equation $\alpha+\beta=\gamma^2$, define a relation $\sim$ as follows: $(\alpha, \beta, \gamma) \sim (\alpha', \beta', \gamma')$ if there exists $\delta \in \mcO_{S}^\ast$ such that $\alpha= \delta^2 \alpha'$, $\beta= \delta^2 \beta'$ and $\gamma= \delta \gamma'$. Using~\cite[Theorem 39]{M22} with $i=2$, the equation $\alpha+\beta=\gamma^2$ with $\alpha, \beta \in \mcO_{S}^\ast$ and $\gamma \in \mcO_{S}$ has only finitely many solutions up to the equivalence $\sim$.
	\end{remark}
	
	The following corollary follows from Theorems~\ref{main result2 for (r,r,p)}, \ref{main result1 for (r,r,p)}.
	%since $2 \nmid c$ iff $\mfP \nmid c$ when $S_{K, 2}= \{\mfP\}$ and $S_{K, p}= \{\mfq\}$.
	\begin{cor}
		\label{cor to main result2}
		\begin{enumerate}
			\item 
			Let $K, \mfP,d$ satisfy the hypothesis of Theorem~\ref{main result2 for (r,r,p)}. Assume $\mfP$ is unique. Then, there exists a constant $V=V_{K,d}>0$ (depending on $K,d$) such that for primes $p>V$, the equation $x^5+y^5=dz^p$ has no non-trivial primitive solution $(a,b,c) \in \mcO_{K}^3$ with $c\neq \pm1$ and $2 \nmid c$.
			\item On the other hand, let $K, \mfP,d$ satisfy the hypothesis of Theorem~\ref{main result1 for (r,r,p)} with $r=5$. Then, there exists a constant $V=V_{K,d}>0$ (depending on $K,d$) such that for primes $p>V$, the equation $x^5+y^5=dz^p$ has no non-trivial primitive solution $(a,b,c) \in \mcO_{K}^3$ with $2 | c$.
		\end{enumerate} 
	\end{cor}
	\begin{remark}
		We can remove the condition $c\neq \pm1$ in Theorem~\ref{main result2 for (r,r,p)} if we assume the modularity of elliptic curves over $K$. In particular, the elliptic curves over $K=\Q$ are modular. Further, $K=\Q$, $\mfP=2\Z$ and $d=\pm 2^m$ with $m \geq 5$ satisfy all the hypotheses of Theorem \ref{main result2 for (r,r,p)}. 
	\end{remark}
	\subsection{Comparison with the works of~\cite{M23}, \cite{KMO24}:}
	We discuss how the results of this article for studying the solutions of the equation $x^r+y^r=dz^p$ with $d \in \mcO_K\setminus \{0\}$ over $K$ are an improvement and generalization of the method used in~\cite{M23} (respectively \cite{KMO24}) to study the solutions of the equation $x^r+y^r=z^p$ over $K$ (respectively $x^r+y^r=dz^p$ with $d \in \Z \setminus\{0\}$ over $\Q$).
	\begin{itemize}
		
		\item Regarding the modularity, \cite{M23} (respectively  \cite{KMO24}) proved that the modularity of the Frey curve corresponding to the equation $x^r+y^r=z^p$ (respectively $x^r+y^r=dz^p$) with $p \gg 0$ over $K^+$ (respectively $\Q^+$). However, we prove the modularity of the Frey curves corresponding to the equations $x^r+y^r=dz^p$ and $x^5+y^5=dz^p$ with $p \gg 0$ over $K^+$ and over $K$, respectively (cf. Theorems~\ref{modularity result for main result1}, \ref{modularity result for main result2}). Whereas the Frey curve in \cite{M23} (respectively \cite{KMO24}) is semistable away from $S_{K^+,2r}$ (respectively  $S_{\Q^+,2r}$), here the Frey curve $E/K^+$ in \eqref{Frey curve result1} (respectively $E/K$ in \eqref{Frey curve result2}) is semistable away from $S_{K^+,2rd}$ (respectively $S_{K,10d}$) (cf. Theorems~\ref{semi stable red of Frey curve},~\ref{reduction away from S_{K,10d}}).
		\item 	Note that our set of primes $S_{K^+, 2rd}$ contains the one in in \cite{M23}, namely $S_{K^+, 2r}$, adding complexity to the problem.
		%	Due to the presence of $d$ in our set $S_{K^+, 2rd}$, the $S_{K^+, 2rd}$-unit equation~\eqref{S_K-unit solution} becomes more complicated than the $S_{K^+, 2r}$-unit equation  in \cite{M23}.          
		%\item 
		To get a contradiction in Step $3$, the proof of~\cite[Theorem 4]{M23} for $x^r+y^r=z^p$ uses some explicit bounds on the solutions of the $S_{K^+,2r}$-unit equation and the proof of \cite[Theorem 1.1]{KMO24} uses the weak Frey-Mazur conjecture. However, the proof of Theorem~\ref{main result1 for (r,r,p)} for $x^r+y^r=dz^p$ uses some explicit bounds on the solutions of $S_{K^+,2rd}$-unit equation to get a contradiction in Step $3$. 
		\item Since the proof of Theorem~\ref{main result1 for (r,r,p)} uses the explicit bound \eqref{assumption for main result1} of the $S_{K^+, 2rd}$-unit equation \eqref{S_K-unit solution}, we are able to give local criteria of $K$ satisfying Theorem~\ref{main result1 for (r,r,p)} when $d$ is a rational prime and $d \in \mcO_K^\ast$ (cf. \S\ref{section for d is a rational prime} and \S\ref{section for d is a unit}). Note that the local criteria of $K$ when $d$ is a rational prime is different from the local criteria of $K$ in \cite{M23}. The same ideas do not continue to hold for studying all the solutions $(a,b,c) \in \mcO_K^3$ to the equation $x^r+y^r=dz^p$ since there are units in the fields $K(\zeta_r)^+$ satisfying \eqref{S_K-unit solution} for any $r \geq 5$. Indeed, for any rational prime $r \geq 5$ with $K:=\Q(\zeta_r)^+$, there are units $\lambda, \mu \in \mcO_K^\ast$ which satisfy (2.2) (cf. \cite[Lemma 18]{SV25} for more details).
		
		\item Note that the Frey curve in \eqref{Frey curve result2} for the equation $x^5+y^5=dz^p$ is different from the Frey curve in \cite{M23} (respectively  \cite{KMO24}) and also it is defined over $K$ instead of $K^+$ (respectively $\Q^+$). Thus the method of proving Theorem~\ref{main result2 for (r,r,p)} is quite different from \cite{M23}, \cite{KMO24}. Also, to arrive at a contradiction in Step 3 outlined above, we use an explicit bound on the solutions of	the equation $\alpha+\beta=\gamma^2$, where $\alpha, \beta \in \mcO_{S_{K,10d}}^\ast$ and $\gamma \in \mcO_{S_{K,10d}}$ instead of the $S_{K^+,2r}$-unit equation considered in \cite{M23}. 
		
	\end{itemize}
	
	% \cite{M23} proved that, for any non-trivial primitive solution $(a,b,c) \in \mcO_{K^+}^3$ to $x^r+y^r=z^p$ with $c$ is even and $p \gg 0$, the Frey elliptic curve is modular over $K^+$. However, for any non-trivial primitive solution $(a,b,c) \in \mcO_{K^+}^3$ to $x^r+y^r=dz^p$ with $c$ is even (respectively $(a,b,c) \in W_K$ to $x^5+y^5=dz^p$) of exponent $p \gg 0$, we prove that the Frey elliptic curve $E/K^+$ in \eqref{Frey curve result1} (respectively $E/K$ in \eqref{Frey curve result2}) is modular (cf. Theorems~\ref{modularity result for main result1}, \ref{modularity result for main result2}).
	%\item Using~\cite[Theorem 2]{FS15 Irred}, for $p \gg 0$, the residual representation $\bar{\rho}_{E,p}$ is irreducible. 

	%This helped us to study the asymptotic solutions $(a,b,c) \in \mcO_{K^+}^3$ to $x^r+y^r=dz^p$ with $2|c$ (respectively  $(a,b,c) \in \mcO_{K}^3$ to $x^5+y^5=dz^p$ with $2 \nmid c$). 

	%However, using the explicit bound \eqref{assumption for main result2} on the solutions of \eqref{eqn for main result2}, we are able to study the solutions $(a,b,c) \in \mcO_K^3$ of $x^5+y^5=dz^p$ with $2 \nmid c$ since the equation \eqref{eqn for main result2} is defined over $K$. This follows from the fact that the Frey elliptic curve for the equation $x^5+y^5=dz^p$ is defined over $K$ (cf. \eqref{Frey curve result2} for the Frey curve).
	
	\section{Proof of Theorem~\ref{main result1 for (r,r,p)}}
	\label{steps to prove main result1}
	
	\subsection{Construction of Frey elliptic curve}
	\label{section for Frey curve}
	In this subsection, we recall the Frey elliptic curve associated with any non-trivial primitive solution to the Diophantine equation $x^r+y^r=dz^p$ from \cite{F15}.
	Recall that $L:=K(\zeta_r)$ and $K^+:=K(\zeta_r+ \zeta_r^{-1})$, where $\zeta_r$ is a primitive $r$th root of unity. Recall that $\mfP \in S_{K+, 2}$ is a prime ideal of $\mcO_{K^+}$ lying above $2$.
	For any prime $r \geq 5$, let $$\phi_r(x,y):= \frac{x^r+y^r}{x+y}= \sum_{i=0}^{r-1}(-1)^ix^{r-1-i}y^i.$$ 
	Then the factorization of the polynomial $\phi_r(x,y)$ over $L$ is as follows.
	\begin{small}
		\begin{equation}
			\label{facored over L}
			\phi_r(x,y)= \prod_{i=1}^{r-1}(x+\zeta_r^{i} y).
		\end{equation}
	\end{small}
	For $0 \leq k \leq \frac{r-1}{2}$, let $f_k(x,y):= (x+\zeta_r^k y) (x+\zeta_r^{-k} y)= x^2+ (\zeta_r^{k}+\zeta_r^{-k})xy+ y^2$. Then $f_k(x,y) \in K^{+}[x,y]$ for all $k$ with $0 \leq k \leq \frac{r-1}{2}$.
	The factorization of the polynomial $\phi_r(x,y)$ over $K^{+}$ is as follows.
	\begin{small}
		\begin{equation}
			\label{facored over K^+}
			\phi_r(x,y)= \prod_{i=1}^{\frac{r-1}{2}} f_k(x,y).
		\end{equation}
	\end{small}
	Since $r\geq 5$, $\frac{r-1}{2} \geq 2$. Fix three integers $k_1, k_2, k_3$ such that $0 \leq k_1 <k_2 <k_3 \leq \frac{r-1}{2}$.
	Let 
	\begin{small}
		%\begin{equation}
		\begin{align}
			\label{eqn for alpha, beta, gamma}
			\alpha= \zeta_r^{k_3}+\zeta_r^{-k_3}-\zeta_r^{k_2}-\zeta_r^{-k_2}, \nonumber\\
			\beta= \zeta_r^{k_1}+\zeta_r^{-k_1}-\zeta_r^{k_3}-\zeta_r^{-k_3},\\
			\gamma= \zeta_r^{k_2}+\zeta_r^{-k_2}-\zeta_r^{k_1}-\zeta_r^{-k_1} \nonumber.
		\end{align}
		%\end{equation}
	\end{small}
	%$$\alpha= \zeta^{k_3}+\zeta^{-k_3}-\zeta^{k_2}-\zeta^{-k_2},$$
	%$$ \beta= \zeta^{k_1}+\zeta^{-k_1}-\zeta^{k_3}-\zeta^{-k_3},$$
	%$$\gamma= \zeta^{k_2}+\zeta^{-k_2}-\zeta^{k_1}-\zeta^{-k_1}.$$
	Clearly, $\alpha, \beta, \gamma \in K^{+}$.
	Then $\alpha f_{k_1}+ \beta f_{k_2}+\gamma f_{k_3}=0$. Let $A(x,y):=\alpha f_{k_1} (x,y)$, $B(x,y):=\beta f_{k_2} (x,y)$, and $C(x,y):=\gamma f_{k_3} (x,y)$. Hence $A(x,y), B(x,y),C(x,y) \in K^{+}[x,y]$.
	
	%\subsubsection{}
	%Let $(a,b,c) \in \mcO_{K^+}^3$ be a non-trivial primitive solution of \eqref{r,r,p over K}.
	%Let $k_1, k_2, k_3 \in \{0,1, \dots, \frac{r-1}{2}\}$ such that $k_1, k_2, k_3$ are distinct. Let $A=\alpha f_{k_1}(a,b)$, $B=\beta f_{k_2}(a,b)$, and $C=\gamma f_{k_3}(a,b)$. Consider the Frey curve 
	%$E:=E_{a,b,c}$ is as follows.
	%\begin{equation}
	%	\label{Frey curve result2}
	%	E_{a,b,c}: Y^2=X(X-A)(X+B).
	%\end{equation}
	%where $\Delta_E=2^4(ABC)^2$, $c_4=2^4(AB+BC+CA)$ and $j_E=2^8 \frac{(AB+BC+CA)^3}{(ABC)^2}$.
	%
	%\subsubsection{}
	Let $(a,b,c) \in \mcO_{K^+}^3$ be a non-trivial primitive solution of the equation \eqref{r,r,p over K} with $\mfP |c$ for some $\mfP \in S_{K+, 2}$. 
	By \eqref{r,r,p over K} and \eqref{facored over K^+}, we have
	\begin{small}
		\begin{equation}
			\label{rel between c and a+b}
			dc^p= (a+b) \prod_{i=1}^{\frac{r-1}{2}} f_k(a,b).	
		\end{equation}
	\end{small}
	
	Since $f_0(a,b)=(a+b)^2$ and $\mfP |c$, we have $\mfP | f_k(a,b)$, for some $k$ with $0 \leq k \leq \frac{r-1}{2}$. Choose $k_1 \in \{0,1, \dots, \frac{r-1}{2}\}$ such that  $\mfP | f_{k_1}(a,b)$. Choose $k_2, k_3 \in \{0,1, \dots, \frac{r-1}{2}\}$ such that $k_1, k_2, k_3$ are distinct. Let $A:=\alpha f_{k_1}(a,b)$, $B:=\beta f_{k_2}(a,b)$, and $C:=\gamma f_{k_3}(a,b)$. Consider the Frey elliptic curve 
	\begin{equation}
		\label{Frey curve result1}
		E:=E_{a,b,c}/K^+: Y^2=X(X-A)(X+B),
	\end{equation}
	where $\Delta_E=2^4(ABC)^2$, $c_4=2^4(AB+BC+CA)$ and $j_E=2^8 \frac{(AB+BC+CA)^3}{(ABC)^2}$.

	\begin{remark}
		We note that the Frey curve in~\eqref{Frey curve result1} depends on $p$ since $A,B,C$ depend on $p$.
	\end{remark}
	
	\subsection{Conductor of the Frey elliptic curve}
	In this subsection, we study the reduction type of the Frey curve given in \eqref{Frey curve result1}.
	Before that, we will prove some lemmas that will be useful for the reduction of the Frey curve. Recall that, $L:=K(\zeta_r)$ and $S_{L, rd}:=\{\mfp \in P_L : \mfp |rd\}.$
	
	\begin{lem}
		\label{coprime lem1}
		Let $(a,b,c) \in \mcO_{L}^3$ be a non-trivial primitive solution of the equation \eqref{r,r,p over K}. Then any two algebraic numbers $a+\zeta_r^ib$ and $a+\zeta_r^jb$ with $0 \leq i <j \leq r-1$ are coprime outside $S_{L,rd}$. 
	\end{lem}
	
	\begin{proof}
		We will prove this by contradiction. Suppose $\mfq \notin S_{L,rd}$ is a prime of $L$ which divides both $a+\zeta_r^ib$ and $a+\zeta_r^jb$ for some $i,j$ with $0 \leq i <j \leq r-1$. This gives $\mfq | (\zeta_r^i-\zeta_r^j)b=  \zeta_r^i(1-\zeta_r^{j-i})b$. Since only primes of $L$ dividing $(1-\zeta_r^{j-i})$ are the primes inside $S_{L,r}$ and $\mfq \notin S_{L,rd}$, we have $\mfq |b$.
		Therefore $\mfq |a$, and hence $\mfq | dc^p$. Since $\mfq \notin S_{L,rd}$, $\mfq \nmid d$ and hence $\mfq |c$, which is a contradiction to $(a,b,c) $ being primitive.
	\end{proof}
	
	\begin{lem}
		\label{coprime for f_k}
		Let $(a,b,c) \in \mcO_{K^+}^3$ be a non-trivial primitive solution of the equation \eqref{r,r,p over K}. Then any two numbers $f_i(a,b)$ and $f_j(a,b)$ in $ \mcO_{K^+}$ with $0 \leq i <j \leq \frac{r-1}{2}$ are coprime outside $S_{K^+,rd}$. The prime factorization of $f_k(a,b) \mcO_K^+$ for $0 \leq k \leq \frac{r-1}{2}$ is as follows:
		
		\begin{equation}
			f_k(a,b) \mcO_K^+= \mfc_k^p \prod_{\mfp \in S_{K^+,rd}} \mfp^{e_{k,p}},
		\end{equation}
		where $\mfc_k \notin S_{K^+,rd}$ is an ideal of $\mcO_K^+$ with $\mfc_k | c\mcO_K^+$ and $e_{k,p} \geq 0$ for $0 \leq k \leq \frac{r-1}{2}$.
	\end{lem}
	
	\begin{proof}
		Recall that $f_k(a,b)= (a+\zeta_r^kb) (a+\zeta_r^{-k}b)$ for $0 \leq k \leq \frac{r-1}{2}$. By Lemma~\ref{coprime lem1}, $f_i(a,b)$ and $f_j(a,b)$ with $0 \leq i <j \leq \frac{r-1}{2}$ are coprime outside the set $S_{L,rd}$. Since $f_k(a,b) \in \mcO_K^+$ for $0 \leq k \leq \frac{r-1}{2}$ and $K^+ \subset L$, $f_i(a,b)$ and $f_j(a,b)$ with $0 \leq i <j \leq \frac{r-1}{2}$ are coprime outside the set $S_{K^+,rd}$.
		
		By \eqref{r,r,p over K} and \eqref{facored over K^+}, we have 
		$$dc^p= (a+b) \prod_{i=1}^{\frac{r-1}{2}} f_k(a,b).$$ 
		The proof of the lemma follows by comparing the prime factorization of both sides of the above equation in $K^+$. Note that here $f_0(a,b)=(a+b)^2$.
	\end{proof}
	
	\begin{cor}
		\label{coprime for A,B,C}
		Let $(a,b,c) \in \mcO_{K^+}^3$ be a non-trivial primitive solution of the equation \eqref{r,r,p over K}. For any integers $k_1, k_2, k_3$ with $0 \leq k_1 <k_2 <k_3 \leq \frac{r-1}{2}$, the algebraic integers  $A=\alpha f_{k_1} (a,b)$, $B=\beta f_{k_2} (a,b)$, and $C=\gamma f_{k_3} (a,b)$ are coprime outside $S_{K^+,rd}$.
	\end{cor}
	\begin{proof}
		By the definition of $\alpha, \beta, \gamma$ in \eqref{eqn for alpha, beta, gamma}, we can write $\alpha, \beta, \gamma$ of the form $\pm \zeta_r^s (1-\zeta_r^t)(1-\zeta_r^u)$, where $t, u$ are not multiples of $r$. Hence, the only primes of $K^+$ dividing $\alpha \beta \gamma$ are the primes in $S_{K^+,r}$. 
		Since $A=\alpha f_{k_1} (a,b)$, $B=\beta f_{k_2} (a,b)$, $C=\gamma f_{k_3} (a,b)$, the proof of the corollary follows by Lemma~\ref{coprime for f_k} since $f_k(a,b)$ are coprime outside $S_{K^+,rd}$ for $0 \leq k \leq \frac{r-1}{2}$.
	\end{proof}
	The following lemma is similar to \cite[Lemma 13]{M23}, and the proof of the lemma follows from Lemma~\ref{coprime for f_k} and Corollary~\ref{coprime for A,B,C}.
	\begin{lem}
		\label{factorisation of A,B,C}
		Let $(a,b,c) \in \mcO_{K^+}^3$ be a non-trivial primitive solution of the equation~\eqref{r,r,p over K}. Let $A,B,C$ be as in Corollary~\ref{coprime for A,B,C}. Then the prime factorization of $A \mcO_K^+$, $B \mcO_K^+$, $C \mcO_K^+$ is as follows.
		\begin{equation}
			\label{ideal factor of A,B,C}
			A\mcO_K^+= \mfc_A^p \prod_{\mfp \in S_{K^+,rd}} \mfp^{\alpha_{p}},\ B\mcO_K^+= \mfc_B^p \prod_{\mfp \in S_{K^+,rd}} \mfp^{\beta_{p}}, \
			C\mcO_K^+= \mfc_C^p \prod_{\mfp \in S_{K^+,rd}} \mfp^{\gamma_{p}},
		\end{equation}
		where $\mfc_A, \mfc_B, \mfc_C \notin S_{K^+,rd}$ are pairwise coprime prime ideals of $K^+$ dividing $c\mcO_K^+$, and $\alpha_{p}, \beta_{p}, \gamma_{p} \geq 0$.
	\end{lem}
	Let $F$ be a field and let $E/F$ be an elliptic curve with conductor $\mfn$. 
	For any prime ideal $\mfq$ of $F$, let $\Delta_\mfq$ be the minimal discriminant of $E$ at $\mfq$. For any rational prime $p$, let 
	
	\begin{equation}
		\label{conductor of elliptic curve}
		\mfm_p:= \prod_{ p|v_\mfq(\Delta_\mfq) \text{ and}\ \mfq ||\mfn} \mfq \text{ and } \mfn_p:=\frac{\mfn}{\mfm_p}.
	\end{equation}
	The following theorem determines the conductor of the Frey elliptic curve $E$ in \eqref{Frey curve result1} and describes the reduction type of $E$ at primes away from $ S_{K^+,2rd}$.
	
	\begin{thm}
		\label{semi stable red of Frey curve}
		Let $(a,b,c) \in \mcO_{K^+}^3$ be a non-trivial primitive solution of the equation \eqref{r,r,p over K} with $\mfP |c$ for some $\mfP \in S_{K+, 2}$, and let $E/K^+$ be the Frey curve given in~\eqref{Frey curve result1}. Then, for all primes $\mfq\in P_{K^+} \setminus S_{K^+,2rd}$, $E/ K^+$ is minimal, semistable at $\mfq$ and satisfies $p | v_\mfq(\Delta_E)$. Let $\mfn$ be the conductor of $E/ K^+$, and let $\mfn_p$ be as in \eqref{conductor of elliptic curve}. Then,
		\begin{equation}
			\label{conductor of E and E' x^2=By^p+Cz^p Type I}
			\mfn=\prod_{\mfP \in S_{K^+,2}}\mfP^{r_\mfP} \prod_{\mfp \in S_{K^+,rd}}\mfp^{f_\mfp} \prod_{\mfq|ABC,\ \mfq \notin S_{K^+,2rd}}\mfq,\ \mfn_p=\prod_{\mfP \in S_{K^+,2}}\mfP^{r_\mfP'} \prod_{\mfp \in S_{K^+,rd}}\mfp^{f_\mfp'},
		\end{equation}
		where $0 \leq r_\mfP^{\prime} \leq r_\mfP \leq 2+6v_\mfP(2)$ for $\mfP |2$ and $ 0 \leq f_\mfp' \leq f_\mfp \leq 2+3v_\mfp(3)$ for $\mfp \nmid 2$.
	\end{thm}
	
	\begin{proof}
		%		Let $\mfq\in P_{K^+}$ with $\mfq \notin S_{K^+,rd}$.
		Recall that $\Delta_E=2^4(ABC)^2$ and $c_4=2^4(AB+BC+CA)$.
		\begin{enumerate}
			\item If $\mfq \nmid \Delta_E$, then $E$ has good reduction at $\mfq$
			and $p | v_\mfq(\Delta_E)=0$. 
			
			\item Recall that $\mfq \nmid 2rd$. If $\mfq|\Delta_E=2^4(ABC)^2$ then $\mfq|ABC$. By Corollary~\ref{coprime for A,B,C}, $\mfq$ divides exactly one of $A$, $B$, $C$ since $(a,b,c)$ is primitive. This gives $\mfq\nmid c_4=2^4(AB+BC+CA)$, so $E$ is minimal and has multiplicative reduction at $\mfq$.
			%\end{itemize}
			Since $v_\mfq(\Delta_E)=2 v_\mfq(ABC)=2p v_\mfq(\mfc_A \mfc_B \mfc_C)$, $p | v_\mfq(\Delta_E)$.
		\end{enumerate} 
		Using the definition of $\mfn_p$ in~\eqref{conductor of elliptic curve}, we have $\mfq \nmid \mfn_p$ for all $\mfq \notin S_{K^+,2rd}$. Finally, for $\mfP \in S_{K^+,2rd}$, the bounds on $r_\mfP$ follow from \cite[Theorem IV.10.4]{S94}.
	\end{proof}
	
	\subsection{Modularity of the Frey elliptic curve}
	In this subsection, we prove the modularity of the Frey curve $E/ K^+$ as in~\eqref{Frey curve result1} for a sufficiently large prime $p$. First, we recall the definition of the modularity of elliptic curves over totally real number fields.
	\begin{dfn}
		Let $F$ be a totally real number field. We say an elliptic curve $E/F$ is modular if there exists a primitive Hilbert modular newform $f$ over $F$ of parallel weight $2$ with rational eigenvalues such that both $E$ and $f$ have the same $L$-function over $F$.
	\end{dfn} 
	Now, we recall the modularity result of Freitas, Le Hung, and Siksek. 
	\begin{thm} \rm(\cite[Theorem 5]{FLHS15})
		\label{modularity result of elliptic curve over totally real}
		Let $F$ be a totally real number field. Then, up to isomorphism over $\bar{F}$, there exist only finitely many elliptic curves over $F$, which are not modular. 
	\end{thm}
	Now, we prove the modularity of the Frey curve $E$ in~\eqref{Frey curve result1} for a sufficiently large prime $p$.
	\begin{thm}
		\label{modularity result for main result1}
		Let $K$ be a totally real number field, $r\geq 5$ be a fixed rational prime and $d\in \mcO_K \setminus \{0\}$. Let $K^+:=K(\zeta_r+ \zeta_r^{-1})$.  Then, there exists a constant $D=D_{K,r,d}>0$ (depending on $K,r,d$) such that for any non-trivial primitive solution $(a,b,c) \in \mcO_{K^+}^3$ to the equation~\eqref{r,r,p over K} of exponent $p >D$ with $\mfP |c$ for some $\mfP \in S_{K^+, 2}$, the Frey elliptic curve $E/{K^+}$ given in~\eqref{Frey curve result1} is modular.  
	\end{thm}
	
	\begin{proof}
		Since $K^+$ is totally real, by Theorem~\ref{modularity result of elliptic curve over totally real}, there are at most finitely many elliptic curves over $K^+$ (up to $\overline{K^+}$ isomorphism) that are not modular. Let $j_1,\ldots,j_s \in K^+$ be the $j$-invariants of these elliptic curves.
		Then the $j$-invariant of the Frey elliptic curve $E$ is given by $ j_E=2^{8} \frac{(AB+BC+CA)^3}{A^2B^2C^2}$. Let $\theta= \frac{-B}{A}$. Then the $j$-invariant becomes $j(\theta) =2^{8}\frac{(\theta^2-\theta+1)^3}{\theta^2(\theta-1)^2}$.
		For each $i=1,2,\ldots,s$, the equation $j(\theta)=j_i$ has at most six solutions in $K^+$. So, there exists 
		$\theta_1, \theta_2, ..., \theta_t \in K^+$ with $t\leq 6s$ such that $E$ is modular for all $\theta \notin\{\theta_1, \theta_2, ..., \theta_t\}$.
		If $\theta= \theta_k$ for some $k \in \{1, 2, \ldots, t \}$, then $\left(\frac{B}{A} \right)=-\theta_k$. Therefore, $\left(\frac{B}{A} \right)\mcO_K^+=\theta_k \mcO_K^+$.
		
		By the construction of the Frey curve in \eqref{Frey curve result1}, $\mfP |f_{k_1}$. This gives $\mfP |A$, hence $\mfP \nmid BC$ by Corollary~\ref{coprime for A,B,C}. By \eqref{ideal factor of A,B,C}, we have  $\left(\frac{B}{A} \right)\mcO_K^+= \mfc_B^p \mfc_A^{-p} \prod_{\mfp \in S_{K^+,rd}} \mfp^{\beta_{p}-\alpha_{p}}$, where $\mfP |  \mfc_A$ and $ \mfP \nmid \mfc_B$. Hence, $v_\mfP(\theta_k)=-pv_\mfP(\mfc_A)$. This gives $p | v_\mfP(\theta_k)$. Finally, the proof of the theorem follows by taking $D=\max\{|v_\mfP(\theta_1)|, \dots, |v_\mfP(\theta_t)| \}$.
	\end{proof}
	
	\subsection{Irreducibility of the mod $p$ Galois representations attached to elliptic curves}
	Let $F$ be a number field and let $E/F$ be an elliptic curve. For any rational prime $p$, let $\bar{\rho}_{E,p} : G_F:=\Gal(\bar{F}/F) \rightarrow \mathrm{Aut}(E[p]) \simeq \GL_2(\F_p)$ be the mod $p$ Galois representation of $G_F$, induced by the action of $G_F$ on the $p$-torsion points $E[p]$ of $E$.
	For any elliptic curve $E$ over $K$, Freitas and Siksek established a criterion for determining the irreducibility of mod $p$ Galois representations	$\bar{\rho}_{E,p}$. More precisely,
	
	\begin{thm} \rm(\cite[Theorem 2]{FS15 Irred})
		\label{irreducibility of mod $P$ representation}
		Let $F$ be a totally real Galois field. Then there exists an effective constant $C_F>0$ (depending on $F$) such that if $p>C_F$ is a prime and $E/F$ is an elliptic curve over $F$ which is semistable at all primes $\mfq$ of $F$ with $\mfq |p$, then $\bar{\rho}_{E,p}$ is irreducible.
	\end{thm}	
	
	\subsection{Level lowering results}
	For any Hilbert modular newform $f$ over a totally real number field $K$ of weight $k$, level $\mfn$ with coefficient field $\Q_f$ and any $\omega \in P_{\Q_f}$, let $\bar{\rho}_{f, \omega}: G_K \rightarrow \GL_2(\F_\omega)$ be the residual Galois representation attached to $f, \omega$.
	In ~\cite{FS15}, Freitas and Siksek gave a level-lowering result.
	\begin{thm} \rm(\cite[Theorem 7]{FS15})
		\label{level lowering of mod $p$ repr}
		%	Let $K$ be a totally real field and 
		Let $E$ be an elliptic curve over a totally real number field $K$ of conductor
		$\mfn$. Let $p$ be a rational prime. Suppose that the following conditions hold:
		\begin{enumerate}
			\item  For $p \geq 5$, the ramification index $e(\mfq /p) < p-1$ for all $\mfq |p$, and $\Q(\zeta_p)^+ \nsubseteq K$;
			\item $E/K$ is modular;
			\item $\bar{\rho}_{E,p}$ is irreducible;
			\item $E$ is semistable at all $\mfq |p$;
			\item $p| v_\mfq(\Delta_\mfq)$ for all $\mfq |p$.
		\end{enumerate}
		Then there exists a Hilbert modular newform $f$ over $K$ of parallel weight $2$, level $\mfn_p$, and some prime $\lambda$ of $\Q_f$ such that $\lambda | p$ and $\bar{\rho}_{E,p} \sim \bar{\rho}_{f,\lambda}$.
	\end{thm}	
	
	\subsection{Eichler-Shimura}
	We now state the Eichler-Shimura conjecture.
	\begin{conj}[Eichler-Shimura]
		\label{ES conj}
		Let $f$ be a Hilbert modular newform over $K$ of parallel weight $2$, level $\mfn$, and with coefficient field $\Q_f= \Q$. Then, there exists an elliptic curve $E_f /K$ with conductor $\mfn$ having the same $L$-function as $f$ over $K$.
	\end{conj}
	In~\cite[Theorem 7.7]{D04}, Darmon proved that Conjecture~\ref{ES conj} holds over $K$, if either 
	$[K: \Q] $ is odd or there exists some prime ideal $\mfq \in P_K$ such that $v_\mfq(\mfn) = 1$.
	In \cite{FS15}, they provided a partial answer to Conjecture~\ref{ES conj} in terms of mod $p$ Galois representations attached to $E$.
	\begin{thm} \rm(\cite[Corollary 2.2]{FS15})
		\label{FS partial result of E-S conj}
		Let $E$ be an elliptic curve over $K$ and $p$ be an odd prime.
		Suppose that $\bar{\rho}_{E,p}$ is irreducible and $\bar{\rho}_{E,p} \sim \bar{\rho}_{f,p}$ for some Hilbert modular newform $f$ over $K$ of parallel weight $2$ and level
		$\mfn$ with rational eigenvalues.
		Let $\mfq \in P_K$ with $\mfq \nmid p$ be such that
		\begin{enumerate}
			\item E has potential multiplicative reduction at $\mfq$ (i.e., $v_\mfq(j_E) <0$);
			\item $p| \# \bar{\rho}_{E,p}(I_\mfq)$;
			\item  $p \nmid \left(\mathrm{Norm}(K/\Q)(\mfq) \pm 1\right)$.
		\end{enumerate}
		Then there exists an elliptic curve $E_f /K$ of conductor $\mfn$ having the same $L$-function as $f$ over $K$.
	\end{thm}
	
	\subsection{Reduction type of the Frey elliptic curve at primes lying above $2$}
	We recall the following basic result, which can be conveniently found in \cite[Lemma 3.4]{FS15}. We use it to determine the types of reduction of the Frey curve $E$ at primes $\mfP \in S_{K^+,2}$. 
	\begin{lem}
		\label{criteria for potentially multiplicative reduction}
		Let $K$ be a totally real number field. Let $E/K$ be an elliptic curve and $p\geq 5$ be a rational prime. For $\mfq \in P_K$ with $\mfq \nmid p$, $E$ has potential multiplicative reduction at $\mfq$ (i.e., $v_\mfq(j_E) <0$) and $p \nmid v_\mfq(j_E)$ if and only if $p | \# \bar{\rho}_{E,p}(I_\mfq)$.
	\end{lem}
	
	%The following lemma is helpful for the types of the reduction of the Frey curve $E$ at $\mfP \in U_{K^+,2}$.
	%\begin{lem} \rm(\cite[Lemma 3.4]{FS15})
	%	\label{3 divides discriminant}
	%	Let $K$ be a totally real number field. Let $E/K$ be an elliptic curve. Let $p\geq 3$ be a prime and $\mfP\in P_K$ be a prime of $K$ lying above $2$. Suppose $E$ has potential good reduction at $\mfP$ (i.e., $v_\mfP(j_E)  \geq 0$). Then $3 \nmid v_\mfP(\Delta_E)$ if and only if $3 | \#\bar{\rho}_{E,p}(I_\mfP)$.
	%\end{lem}
	
	%	In the following lemma, we will determine the type of reduction of the Frey curve $E$ at primes $\mfP \in S_{K^+,2}$. 
	\begin{lem}
		\label{reduction on T and S}
		Let $\mfP \in S_{K^+,2}$. Let $(a,b,c) \in \mcO_{K^+}^3$ be a non-trivial primitive solution of the equation \eqref{r,r,p over K} of prime exponent $p > 4v_\mfP(2)$ with $\mfP |c$, and let $E$ be the Frey curve given in~\eqref{Frey curve result1}.  Then $\ v_\mfP(j_E) < 0$ and $p \nmid v_\mfP(j_E)$, equivalently $p | \#\bar{\rho}_{E,p}(I_\mfP)$.
	\end{lem}
	
	\begin{proof}
		Recall that, $\Delta_E=2^4(ABC)^2$ and $j_E=2^8 \frac{(AB+BC+CA)^3}{(ABC)^2}$.
		By the construction of the Frey curve given in~\eqref{Frey curve result1} and by Corollary~\ref{coprime for A,B,C}, we have $\mfP |A$ and $\mfP \nmid BC$. This gives $v_\mfP(j_E)=8 v_\mfP(2)-2 v_\mfP(A)=8 v_\mfP(2)-2p v_\mfP(\mfc_A)$, where $\mfc_A$ is as in Lemma~\ref{factorisation of A,B,C}. Since $v_\mfP(\mfc_A) >0$ and $p > 4v_\mfP(2)$, it follows that $\ v_\mfP(j_E) < 0$ and $p \nmid v_\mfP(j_E)$. Finally, by Lemma~\ref{criteria for potentially multiplicative reduction}, we have $p | \#\bar{\rho}_{E,p}(I_\mfP)$.
		%	\begin{enumerate}
			%		\item By the construction of Frey curve given in~\eqref{Frey curve result1}, $\mfP |A$ and $\mfP \nmid BC$. This gives $v_\mfP(j_E)=8 v_\mfP(2)-2 v_\mfP(A)=8 v_\mfP(2)-2p v_\mfP(\mfc_A)$, where $\mfc_A$ is as in Lemma~\ref{factorisation of A,B,C}. Since $v_\mfP(\mfc_A) >0$ and $p > 4v_\mfP(2)$, $\ v_\mfP(j_E) < 0$ and $p \nmid v_\mfP(j_E)$.
			%	    
			%	    \item We will prove this part in two different cases. If $\mfP | ABC$, then by Corollary~\ref{coprime for A,B,C}  $\mfP$ divides exactly one of $A$, $B$, $C$. Similar to (1), we have $\ v_\mfP(j_E) < 0$ and $p \nmid v_\mfP(j_E)$. By Lemma~\ref{criteria for potentially multiplicative reduction}, we have $p | \#\bar{\rho}_{E,p}(I_\mfP)$.
			%	    
			%	    If $\mfP \nmid ABC$, then  $ v_\mfP(j_E)= 8 v_\mfP(2)+ 3 v_\mfP(AB+BC+CA) >0$. Now, $v_\mfP(\Delta_E)= 4 v_\mfP(2)$. Since $\mfP \in U_{K^+,2}$, $3 \nmid v_\mfP(\Delta_E)$. Hence the proof follows from Lemma~\ref{3 divides discriminant}.
			%	\end{enumerate}
	\end{proof}
	
	\subsection{An auxiliary result for Theorem~\ref{main result1 for (r,r,p)}}
	%In this section, we prove Theorem~\ref{main result1 for (r,r,p)}. 
	The next result is a key ingredient in the proof of Theorem~\ref{main result1 for (r,r,p)} and its proof follows closely  \cite[Theorem 9]{FS15} and \cite[Theorem 28]{M23}.
	%, and it's proof is similar to \cite[Theorem 28]{M23} and \cite[Theorem 4.8]{KS23 Diophantine1}.
	\begin{prop}
		\label{auxilary result for main result1}
		Let $K$ be a totally real number field, $r\geq 5$ be a fixed rational prime and $d\in \mcO_K \setminus \{0\}$. Let $K^+:=K(\zeta_r+ \zeta_r^{-1})$. Then, there is a constant $V:=V_{K,r,d}>0$ (depending on $K,r,d$) such that the following holds.
		Let $(a,b,c) \in \mcO_{K^+}^3$ be a non-trivial primitive solution to the equation \eqref{r,r,p over K} of exponent $p>V$ with $\mfP |c$ for some $\mfP \in S_{K+, 2}$, and let $E/{K^+}$ be the Frey curve given in~\eqref{Frey curve result1}. Then, there exists an elliptic curve $E^\prime/K^+$ such that:
		\begin{enumerate}
			\item $E^\prime/K^+$ has good reduction away from $S_{K^+,2rd}$; 
			\item $E^\prime/K^+$ has full $2$-torsion points i.e. $E'(\bar{\Q})[2] \subset E'(K^+)$;
			\item We have $\bar{\rho}_{E,p} \sim \bar{\rho}_{E^\prime,p}$ (as $G_{K^+}$ modules) and  $v_\mfP(j_{E^\prime})<0$.
		\end{enumerate}
	\end{prop}
	\begin{proof}
		By Theorem~\ref{semi stable red of Frey curve}, the Frey curve $E$ is semistable away from the set $S_{K^+,2rd}$. By Theorem~\ref{modularity result for main result1}, the Frey curve $E$ is modular for primes $p\gg0$. By Theorem~\ref{irreducibility of mod $P$ representation} and replacing $K$ with its Galois closure, we conclude that $\bar{\rho}_{E,p}$ is irreducible for $p \gg 0$.
		
		By Theorem~\ref{level lowering of mod $p$ repr}, there exists a Hilbert modular newform $f$ of parallel weight $2$, level $\mfn_p$ and some prime $\lambda$ of $\Q_f$ such that $\lambda | p$ and $\bar{\rho}_{E,p} \sim \bar{\rho}_{f,\lambda}$ for $p \gg 0$. 
		By allowing $p$ to be sufficiently large, we can assume $\Q_f=\Q$. This follows from~\cite[Proposition 15.4.2]{C07} (cf.~\cite[\S 4]{FS15} for more details).
		
		%	 {\color{red}
			%	 If Conjecture~\ref{ES conj} holds}, then by Theorem~\ref{partial result to ES conj}, there exists an elliptic curve $E_f/K$ of conductor $\mfn_p$ having the same $L$-function as of $f$. If $T_K\neq \varphi$, then for any 
		Now, by Lemma~\ref{reduction on T and S}, $E$ has potential multiplicative reduction at $\mfP$ and $p | \#\bar{\rho}_{E,p}(I_\mfP)$
		for $p \gg 0$. By Theorem~\ref{FS partial result of E-S conj}, there exists an elliptic curve $E_f$ of conductor $\mfn_p$ such that $\bar{\rho}_{E,p} \sim \bar{\rho}_{E_f,p}$ for $p \gg 0$ (after excluding the primes $p \mid \left( \text{Norm}(K/\Q)(\mfP) \pm 1 \right)$). Let $V:=V_{K,r,d}$ be the maximum of all the above lower bounds of $p$. Then $\bar{\rho}_{E,p} \sim \bar{\rho}_{E_f,p}$ for $p>V$.
		
		Since the conductor of $E_f$ is $\mfn_p$ given in \eqref{conductor of E and E' x^2=By^p+Cz^p Type I}, $E_f$ has good reduction away from $S_{K^+,2rd}$. The Frey curve $E/K^+$ has full $2$-torsion points given by $\{O, (0,0), (0, A), (0,B)\}$, where $A,B \in K^+$ are as in \S\ref{section for Frey curve}.
		After enlarging $V$ by an effective amount and by possibly replacing $E_f$ with an isogenous curve, say $E^\prime$, we will find that $E^\prime/ K^+$ has full $2$-torsion and  $\bar{\rho}_{E,p} \sim\bar{\rho}_{E^\prime,p}$. 
		This follows from~\cite[Proposition 15.4.2]{C07} and the fact that $E(\bar{\Q})[2] \subset E(K^+)$  (cf.~\cite[\S 4]{FS15} for more details). Finally, since $E_f$ is isogenous to $E^\prime$, it follows that $E^\prime$ has good reduction away from $S_{K^+,2rd}$.
		
		Using Lemma~\ref{reduction on T and S}, we have $p |\# \bar{\rho}_{E,p}(I_\mfP)= \# \bar{\rho}_{E^\prime,p}(I_\mfP)$. By Lemma~\ref{criteria for potentially multiplicative reduction}, we get $v_\mfP(j_{E^\prime})<0$.
		%\end{enumerate}
		This completes the proof of the proposition.
	\end{proof}
	
	We are now ready to complete the proof of Theorem~\ref{main result1 for (r,r,p)}. The proof follows closely \cite[Theorem 3]{FS15} and \cite[Theorem 4]{M23}.
	%, and the proof of this theorem is similar to~\cite[Theorem 3.3]{KS23 Diophantine1} and~\cite[Theorem 3]{FS15}.
	\begin{proof}[Proof of Theorem~\ref{main result1 for (r,r,p)}]
		Let $V=V_{K,r,d}$ be as in Proposition~\ref{auxilary result for main result1}, and let $(a,b,c) \in \mcO_{K^+}^3$ be a non-trivial primitive solution to the equation \eqref{r,r,p over K} of exponent $p>V$ with $\mfP |c$, where $\mfP \in S_{K+, 2}$ as in Theorem~\ref{main result1 for (r,r,p)}. By Proposition~\ref{auxilary result for main result1}, there exists an elliptic curve $E^\prime/K^+$ having full $2$-torsion.
		%		By \cite[Lemma-15(i)]{M22}, the elliptic curve elliptic curve $E^\prime/K$ has a model 
		Hence, $E^\prime$ has a model of the form $E^\prime: Y^2 = (X-e_1)(X-e_2)(X-e_3)$, where $e_1, e_2, e_3$ are distinct and their cross ratio $\lambda= \frac{e_3-e_1}{e_2-e_1} \in \mathbb{P}^1(K^+)-\{0,1,\infty\}$. Then,  $E^\prime$ is isomorphic (over $\overline{K^+}$) to an elliptic curve $E_\lambda$ in the Legendre form:
		\begin{small}		$$E_\lambda : y^2 = x(x - 1)(x-\lambda) \text{ for } \lambda \in \mathbb{P}^1(K^+)-\{0,1,\infty\} \text{ with}$$ 
			\begin{equation}
				\label{j'-invariant of Legendre form}
				j_{E^\prime} = j(E_\lambda)= 2^8\frac{(\lambda^2-\lambda+1)^3}{\lambda^2(1-\lambda)^2}.
			\end{equation}
		\end{small}
		Since $E^\prime$ has good reduction away from $S_{K^+,2rd}$, $j_{E^\prime} \in \mcO_{S_{K^+,2rd}}$. Using \eqref{j'-invariant of Legendre form}, we have $\lambda, 1-\lambda \in \mcO^*_{S_{K^+,2rd}}$.
		%		 (cf. \cite[Page 40]{KS23 Diophantine1}, \cite[Page 1407]{FS15} for more details). 
		Hence $(\lambda,\ \mu:=1-\lambda)$ is a solution of the $S_{K^+,2rd}$-unit equation~\eqref{S_K-unit solution}.
		Rewriting \eqref{j'-invariant of Legendre form} in terms of $\lambda, \ \mu$, we have
		\begin{small}
			\begin{equation}
				\label{j' in terms of lambda and mu}
				j_{E^\prime}= 2^8\frac{(1-\lambda \mu)^3}{(\lambda \mu)^2}.
			\end{equation}
		\end{small}	
		Using \eqref{assumption for main result1}, we have $t:=\max \left\{|v_\mfP(\lambda)|,|v_\mfP(\mu)| \right\}\leq 4v_\mfP(2)$. 
		%		Let $t= \max \left( |v_\mfP(\lambda)|,|v_\mfP(\mu)| \right )$.
		If $t=0$, then $v_\mfP(\lambda)= v_\mfP(\mu)=0$, and hence $v_\mfP(j_{E^\prime})\geq 8v_\mfP(2)>0$. 
		If $t>0$, then either $v_\mfP(\lambda)=v_\mfP(\mu)=-t$, or $v_\mfP(\lambda)=t$ and $ v_\mfP(\mu)=0$, or $v_\mfP(\lambda)=0$ and $v_\mfP(\mu)=t$, since $\lambda + \mu =1$. This gives $v_\mfP(\lambda \mu)=-2t$ or $t$, and $v_\mfP(j_{E^\prime})\geq 8v_\mfP(2)-2t \geq 0$. In all cases, we have $v_\mfP(j_{E^\prime})\geq 0 $, which contradicts Proposition~\ref{auxilary result for main result1}(3). This completes the proof of Theorem~\ref{main result1 for (r,r,p)}.
	\end{proof}

	\section{Proof of Theorem~\ref{main result2 for r,r,p}}
	\label{steps to prove main result2}
	\subsection{Construction of Frey elliptic curve for $x^5+y^5=dz^p$}
	For any non-trivial primitive solution $(a, b, c)\in \mcO_K^3$ to the Diophantine equation $x^5+y^5=dz^p$, consider the Frey elliptic curve 
	\begin{equation}
		\label{Frey curve result2}
		E:=E_{a,b,c}/K : y^2= x^3-5(a^2+b^2)x^2+5\phi_5(a,b)x,
	\end{equation}
	where $\phi_5(a,b):=\frac{a^5+b^5}{a+b}, \ c_4=2^4 5 (5(a^2+b^2)^2-3\phi_5(a,b)),\\ \Delta_E=2^4 5^3 (a+b)^2(a^5+b^5)^2$ and $ j_E=2^{8} \frac{(5(a^2+b^2)^2-3\phi_5(a,b))^3}{(a+b)^2(a^5+b^5)^2}$ (cf. \cite[\S2]{B07} for the Frey curve). 
	Note that the Frey curve of the form \eqref{Frey curve result2} works specifically for $r=5$.
	\subsection{Modularity of the Frey curve}
	In this subsection, we prove the modularity of the Frey curve in \eqref{Frey curve result2} for a sufficiently large prime $p$.
	\begin{thm}
		\label{modularity result for main result2}
		Let $K$ be a totally real number field and let $d\in \mcO_K \setminus \{0\}$. Then, there exists a constant $D=D_{K,d}>0$ (depending on $K,d$) such that for any non-trivial primitive solution $(a,b,c) \in \mcO_{K}^3$ to the equation $x^5+y^5=dz^p$ of exponent $p >D$ with $c \neq \pm1$, the Frey elliptic curve $E/K$ given in~\eqref{Frey curve result2} is modular.  
	\end{thm}
	\begin{proof}
		By Theorem~\ref{modularity result of elliptic curve over totally real}, there exist only finitely many elliptic curves over $K$ up to $\bar{K}$ isomorphism that are not modular. Let $j_1,\ldots,j_s \in K$ be the $j$-invariants of those elliptic curves. Then 
		$$ j_E=2^{8} \frac{(5(a^2+b^2)^2-3\phi_5(a,b))^3}{(a+b)^2(a^5+b^5)^2} =2^{8}\frac{\left(5(1+\theta^2)^2(1+\theta)-3(1+\theta^5) \right)^3}{(1+\theta)^5(1+\theta^5)^2},$$
		where $\theta= \frac{b}{a}$.
		For each $i=1,2,\ldots,s$, the equation $j_E=j_i$ has at most fifteen solutions in $K$. So, there exists 
		$\theta_1, \theta_2, ..., \theta_t \in K$ with $t\leq 15s$ such that $E$ is modular for all $\theta \notin\{\theta_1, \theta_2, ..., \theta_t\}$.
		If $\theta= \theta_k$ for some $k \in \{1, 2, \ldots, t \}$, then $\frac{b}{a} =\theta_k$. This gives $d\frac{c^p}{a^5}=1+\theta_k^5$, and hence  $c^p=\frac{a^5(1+\theta_k^5)}{d}=\alpha_k$ (say). Since $c \neq \pm1$ and $K$ is totally real by assumption, there is at most one prime $p= p_k$ (say) satisfying the equation $c^p= \alpha_k$.
		%The above equation determines $p$ uniquely, which we denote by $p_k$. Otherwise, $c$ is a root of unity inside $K$. Since $K$ is totally real, $c = \pm 1$, which is not possible.
		Hence the proof follows by considering $p >D:=\max \{p_1,...,p_t\}$.
	\end{proof}
	
	\subsection{Reduction type}
	In this subsection, we will study the reduction type of Frey curve $E:= E_{a,b,c}$ in~\eqref{Frey curve result2} at primes $\mfq \in P_K$. To do this, we need the following lemma.
	\begin{lem}
		\label{coprime of a+b, phi(a,b)}
		Let $(a,b,c) \in \mcO_K^3$ be a non-trivial primitive solution to the equation  $x^5+y^5=dz^p$, and let $\mfq \in P_K$. 
		\begin{enumerate}
			%			\item Then $a,b,c$ are pairwise coprime outside $d$.
			\item If $\mfq |a+b$, then $\phi_5(a,b) \equiv 5a^2b^2 \pmod{\mfq^2}$.
			\item $a+b$ and $\phi_5(a,b)$ are coprime outside $S_{K,5}$.
		\end{enumerate}
	\end{lem}
	
	\begin{proof}
		Let $\mfq |a+b$. Then 
		$(a+b)^2 \equiv 0 \pmod {\mfq^2}$. So $$\phi_5(a,b) = (a^2+b^2)^2-ab(a^2+b^2+ab) \equiv 4a^2b^2+a^2b^2 = 5a^2b^2 \pmod {\mfq^2}.$$  This completes the proof of (1).
		
		Let $\mfq \in P_K \setminus S_{K,5}$. Suppose $\mfq$ divide both $a+b$ and $\phi_5(a,b)$. By (1), we have $\phi_5(a,b) \equiv 5a^2b^2 \pmod{\mfq^2}$. Since $\mfq |\phi_5(a,b)$ and $\mfq\notin S_{K,5}$, we have $\mfq |ab$. Since $\mfq |a+b$,  $\mfq$ divides both $a$ and $b$, which is not possible by Remark~\ref{coprime of a,b,c} since $(a,b,c)$ is primitive. This completes the proof of (2).
	\end{proof}
	
	The following lemma characterizes the type of reduction of the Frey curve $E:= E_{a,b,c}$ in~\eqref{Frey curve result2} at primes $\mfq$ away from $S_{K, 10d}$.
	\begin{lem}
		\label{reduction away from S_{K,10d}}
		Let $(a,b,c) \in \mcO_K^3$ be a non-trivial primitive solution to the equation  $x^5+y^5=dz^p$. Let $E$ be the Frey curve attached to $(a,b,c)$ as in~\eqref{Frey curve result2}. Then at all primes  $\mfq \in  P_K\setminus S_{K, 10d}$, $E$ is minimal, semistable at $\mfq$ and $p | v_\mfq(\Delta_E)$. Let $\mfn$ be the conductor of $E$ and $\mfn_p$ be as in \eqref{conductor of elliptic curve}. Then,
		\begin{equation}
			\label{conductor of E result2}
			\mfn=\prod_{\mfp \in S_{K, 10d}}\mfp^{r_\mfp} \prod_{\mfq|a^5+b^5,\ \mfq \notin S_{K, 10d}}\mfq, \ 
			\mfn_p=\prod_{\mfp \in S_{K, 10d}}\mfp^{r_\mfp^{\prime}},
		\end{equation}
		where $0 \leq r_\mfp^{\prime} \leq r_\mfp \leq 2+3v_\mfp(3)+ 6v_\mfp(2)$ for $\mfp \in S_{K,10d}$.
	\end{lem}
	\begin{proof}
		Here $c_4=2^4 5 (5(a^2+b^2)^2-3\phi_5(a,b)),\ \Delta_E=2^4 5^3 (a+b)^2(a^5+b^5)^2$. 
		Let $\mfq \notin  S_{K, 10d}$ be a prime of $K$. 
		\begin{enumerate}
			\item  If $\mfq \nmid \Delta_E$, then $E$ has good reduction at $\mfq$ and trivially $p |v_\mfq(\Delta_E)$.
			
			\item Recall that $\mfq \nmid 10d$. If $\mfq|\Delta_E$, then $\mfq|a^5+b^5$.
			Then either $\mfq |a+b$ or $\mfq | \phi_5(a,b)$. Note that here $\mfq \nmid ab$. Otherwise, if $\mfq |ab$, then $\mfq |a$ or  $\mfq |b$. Since  $\mfq|a^5+b^5$, $\mfq$ divide both $a$ and $b$ which is not possible by Remark~\ref{coprime of a,b,c}.
			\begin{itemize}
				\item If $\mfq |a+b$, then by Lemma~\ref{coprime of a+b, phi(a,b)}, we have $\phi_5(a,b) \equiv 5a^2b^2 \pmod {q^2}$ and $\mfq \nmid \phi_5(a,b)$. This gives $c_4 \equiv 2^4 . 5 (5 . 4 a^2b^2-3 . 5a^2b^2)= 2^45^2a^2b^2 \pmod {\mfq^2}$.  Hence $\mfq\nmid c_4$.
				Here $v_\mfq(a+b)= v_\mfq(a^5+b^5)=v_\mfq(dc^p)= pv_\mfq(c)$. Now $v_\mfq(\Delta_E)= 4 v_\mfq(a+b)+ 2 v_\mfq(\phi_5(a,b)) = 4pv_\mfq(c)$. This gives $p |v_\mfq(\Delta_E)$.
				
				\item If $\mfq | \phi_5(a,b)= (a^2+b^2)^2-ab(a^2+b^2+ab)$, then by Lemma~\ref{coprime of a+b, phi(a,b)}, we have $\mfq \nmid a+b$. Further, as $\mfq \nmid ab$, we deduce that $\mfq \nmid a^2+b^2$. Hence we conclude that $\mfq\nmid c_4$.
				Observe that $v_\mfq(\phi_5(a,b))= v_\mfq(a^5+b^5)=v_\mfq(dc^p)= pv_\mfq(c)$. Now $v_\mfq(\Delta_E)= 2 v_\mfq(\phi_5(a,b)) = 2pv_\mfq(c)$. This gives us that  $p |v_\mfq(\Delta_E)$.
				
			\end{itemize} 
			Thus, in either case we have, $\mfq\nmid c_4$, hence $E$ is minimal and has multiplicative reduction at $\mfq$. 
		\end{enumerate}
		Finally, for $\mfp \in S_{K, 10d}$, the bounds on $r_\mfp$ follow from \cite[Theorem IV.10.4]{S94}.
	\end{proof}
	The following lemma determines the type of reduction of $E_{a,b,c}$ in~\eqref{Frey curve result2} at $\mfP \in S_{K, 2}$.
	\begin{lem}
		\label{reduction on T and S main result2}
		Let $(a,b,c) \in \mcO_K^3$ be a non-trivial primitive solution to the equation  $x^5+y^5=dz^p$ with exponent $p > 4|v_\mfP(d)- 2v_\mfP(2)|$, and let $E$ be the Frey curve attached to $(a,b,c)$ as in~\eqref{Frey curve result2}.
		For $\mfP \in S_{K, 2}$, if $v_\mfP(d)> 4v_\mfP(2)$ and $\mfP \nmid c$, then $\ v_\mfP(j_E) < 0$ and $p \nmid v_\mfP(j_E)$. In particular, by Lemma~\ref{criteria for potentially multiplicative reduction}, we have  $p | \#\bar{\rho}_{E,p}(I_\mfP)$.
	\end{lem} 
	%\begin{proof}
	{\it Proof.}
	Recall that $ j_E=2^{8} \frac{(5(a^2+b^2)^2-3\phi_5(a,b))^3}{(a+b)^2(a^5+b^5)^2}$. Since $\mfP|d$, it follows that $\mfP |a^5+b^5$. Then either $\mfP |a+b$ or $\mfP | \phi_5(a,b)$. 
	%		 Simillar to Lemma~\ref{reduction away from S_K'}, we have $\mfP \nmid ab$.
	%		 Otherwise, if $\mfq |ab$, then  $\mfq |a$ or  $\mfq |b$. Since  $\mfq|a^5+b^5$, $\mfq$ divide both $a$ and $b$ which is not possible by Remark~\ref{coprime of a,b,c}.
	\begin{itemize}
		\item If $\mfP |a+b$, then by Lemma~\ref{coprime of a+b, phi(a,b)}, we have $\mfP \nmid \phi_5(a,b)$. As $\mfP|a+b$, we have that $\mfP |a^2+b^2$. Since $\mfP \nmid c$, it follows that $v_\mfP(d)= v_\mfP(dc^p)= v_\mfP(a^5+b^5)= v_\mfP(a+b)$. Then $v_\mfP(j_E)= 8v_\mfP(2)- 2v_\mfP(a+b)-2v_\mfP(a^5+b^5)= 8v_\mfP(2)-4v_\mfP(d)$. As $v_\mfP(d)> 4v_\mfP(2)$, we get that $v_\mfP(j_E) <0$. Since $p > 4v_\mfP(d)- 8v_\mfP(2)$, we obtain that $p \nmid v_\mfP(j_E)$. 
		
		\item If $\mfP| \phi_5(a,b)$, then by Lemma~\ref{coprime of a+b, phi(a,b)}, we have that $\mfP \nmid a+b$ and hence $\mfP \nmid a^2+b^2$. So, $v_\mfP(d)= v_\mfP(a^5+b^5)= v_\mfP(\phi_5(a,b))$. Then $v_\mfP(j_E)=  8v_\mfP(2)- 2v_\mfP(d)$. As $v_\mfP(d)> 4v_\mfP(2)$, we have that $v_\mfP(j_E) <0$. Since $p > 2v_\mfP(d)- 8v_\mfP(2),$ it follows that $p \nmid v_\mfP(j_E)$.
	\end{itemize}
	This completes the proof. \qed
	% Finally, by Lemma~\ref{criteria for potentially multiplicative reduction}, we have  $p | \#\bar{\rho}_{E,p}(I_\mfP)$.
	% \end{proof}

%\subsection{Proof of Theorem~\ref{main result2 for r,r,p}}
%The following theorem is very helpful in the proof of Theorem~\ref{main result2 for r,r,p}, and its proof is similar to Theorem~\ref{auxilary result for main result1}.
\subsection{An auxiliary result for Theorem~\ref{main result2 for (r,r,p)}}
The next result is a key ingredient in the proof of Theorem~\ref{main result2 for (r,r,p)}. Recall that for any $(a,b,c) \in W_K$, we have $c\neq \pm1$ and $\mfP \nmid c$ for all $\mfP \in S_{K,2}$.
\begin{prop}
	\label{auxilary result for main result2}
	Let $K$ be a totally real number field and let $d\in \mcO_K \setminus \{0\}$. Then, there is a constant $V:=V_{K,d}>0$ (depending on $K,d$) such that the following holds:
	Let $(a,b,c) \in W_K$ be a solution to the equation $x^5+y^5=dz^p$ of exponent $p>V$. Let $E/K$ be the Frey curve given in~\eqref{Frey curve result1}. Then, there exists an elliptic curve $E^\prime/K$ such that:
	\begin{enumerate}
		\item $E^\prime/K$ has good reduction away from $S_{K, 10d}$;
		\item $E^\prime/K$ has a non-trivial $2$-torsion point i.e. $ E'(K)[2] \neq \{O\}$;
		\item $\bar{\rho}_{E,p} \sim\bar{\rho}_{E^\prime,p}$ (as $G_{K}$ modules);
		\item Let $\mfP \in S_{K, 2}$. If $v_\mfP(d)> 4v_\mfP(2)$, then $v_\mfP(j_{E^\prime})<0$.
	\end{enumerate}
\end{prop}
{\it Proof.}
We proceed as in the proof of Proposition~\ref{auxilary result for main result1}. By Theorems~\ref{irreducibility of mod $P$ representation},~\ref{level lowering of mod $p$ repr},~\ref{FS partial result of E-S conj},~\ref{modularity result for main result2}, Lemmas~\ref{reduction away from S_{K,10d}},~\ref{reduction on T and S main result2}, there exists a constant $V:=V_{K,d}>0$ (depending on $K,d$) and an elliptic curve $E_f/K$ of conductor $\mfn_p$, where $\mfn_p$ is defined in \eqref{conductor of E result2}, such that $\bar{\rho}_{E,p} \sim \bar{\rho}_{E_f,p}$ for $p>V$.

%\begin{enumerate}
%\item 
The Frey curve $E$ has the $K$-rational $2$-torsion point $(0,0)$. After enlarging $V$ by an efficient amount and by possibly replacing $E_f$ with an isogenous curve, say $E^\prime/K$ such that $ E'(K)[2] \neq \{O\}$ and  $\bar{\rho}_{E,p} \sim\bar{\rho}_{E^\prime,p}$. This follows from~\cite[Proposition 15.4.2]{C07} and the fact $ E(K)[2] \neq \{O\}$ (cf.~\cite[page 10]{KS23 Diophantine2} for more details).   

%\item 
Since the conductor of $E_f$ is $\mfn_p$, where $\mfn_p$ is given by a product of powers of primes in $S_{K, 10d}$ (see \eqref{conductor of E result2}) and $E_f$ is isogenous to $E^\prime$, $E'/K$ has good reduction away from $S_{K, 10d}$.
%Since $E_f$ is isogenous to $E'$ and $\bar{\rho}_{E,p} \sim \bar{\rho}_{E_f,p}$, we have $\bar{\rho}_{E,p} \sim\bar{\rho}_{E^\prime,p}$. 
Let $\mfP \in S_{K, 2}$ with $v_\mfP(d)> 4v_\mfP(2)$. Then by Lemma~\ref{reduction on T and S main result2}, we have $p | \# \bar{\rho}_{E,p}(I_\mfP)=\# \bar{\rho}_{E^\prime,p}(I_\mfP)$. Finally, by Lemma~\ref{criteria for potentially multiplicative reduction}, we have $v_\mfP(j_{E^\prime})<0$. \qed
%\end{enumerate}

We are now ready to prove Theorem~\ref{main result2 for (r,r,p)}. 
%We would like to mention that the last paragraph of the proof below follows exactly as in the proof of \cite[Theorem 3]{M22} and we include it here for the sake of completeness. 
\begin{proof}[Proof of Theorem~\ref{main result2 for r,r,p}]
	Let $V=V_{K,d}$ be as in Proposition~\ref{auxilary result for main result2}, and let $(a,b,c)\in W_K$ be a solution to the equation $x^5+y^5=dz^p$ with exponent $p>V$. By Proposition~\ref{auxilary result for main result2}, there exists an elliptic curve $E^\prime/K$ having a non-trivial $2$-torsion point.
	%		By \cite[Lemma-15(i)]{M22}, the elliptic curve elliptic curve $E^\prime/K$ has a model 
	Therefore $E^\prime/K$ has a model of the form
	\begin{equation}
		\label{j invariant of E'}
		E^\prime:y^2=x^3+a'x^2+b'x
	\end{equation}		
	for some $a',b' \in K$ with $b' \neq 0$ and $a'^2 \neq 4b'$. Indeed, the conditions $b' \neq 0$ and $a'^2 \neq 4b'$ comes from the fact that the discriminant of $E'$ is non-zero.
	Here $j_{E^\prime}=2^8 \frac{({a'}^2-3b')^3}{{b'}^2(a'^2-4b')}$. By Proposition~\ref{auxilary result for main result2}, $E^\prime$ has good reduction away from $ S_{K, 10d}$, and hence $j_{E^\prime} \in \mcO_{ S_{K, 10d}}$. 
	
	Define $\lambda := \frac{{a'}^2}{b'}$ and $ \mu := \lambda-4$. We say a fractional ideal $I$ in $K$ is a $S$-ideal if $I$ is generated by primes lying inside $S$. By ~\cite[Lemmas 16(i), 17(i)]{M22}, we have $\lambda, \mu \in \mcO_{S_{K, 10d}}^\ast$ and $ \lambda  \mcO_K=I^2J$, where $I$ is a fractional ideal and $J$ is a $S_{K, 10d}$-ideal. This gives $ [I]^2 =1$ in $\Cl_{S_{K, 10d}}(K)$. Since $\Cl_{S_{K, 10d}}(K)[2]=1$, we have $I=\gamma I'$ for some $\gamma \in \mcO_K$ and a $S_{K, 10d}$-ideal $I'$. Therefore, $\lambda \mcO_K=\gamma^2{I'}^2J$ and hence $(\frac{\lambda}{\gamma^2}) \mcO_K$ is a $S_{K, 10d}$-ideal. Therefore, $u:=\frac{\lambda}{\gamma^2}\in \mcO_{S_{K, 10d}}^\ast$.
	Dividing the equation $\mu +4=\lambda$ by $u$, we have 
	$\alpha +\beta =\gamma^2$, where $\alpha= \frac{\mu}{u} \in \mcO_{S_{K, 10d}}^\ast$ and $\beta =\frac{4}{u} \in \mcO_{S_{K, 10d}}^\ast$. This gives $ \alpha \beta^{-1}=\frac{\mu}{4}$. Let $\mfP \in S_{K, 2}$ be the prime satisfying the hypothesis~\eqref{assumption for main result2}. Then, we get
	\begin{equation*}
		\label{inequality for valution of mu}
		-4v_\mfP(2) \leq v_\mfP(\mu) \leq 8 v_\mfP(2).
	\end{equation*}
	Now, the proof follows from \cite[Theorem 3]{M22} and we include it for the sake of completeness.
	%where $\mfP \in S_{K, 2}$ as in Theorem~\ref{main result2 for (r,r,p)}.
	Writing $j_{E^\prime}$ in terms of $\mu$, we have $j_{E^\prime}= 2^8 \frac{(\mu+1)^3}{\mu}$. Hence $v_\mfP(j_{E^\prime})=8v_\mfP(2)+3 v_\mfP(\mu +1)-v_\mfP(\mu)$. 
	For $-4v_\mfP(2) \leq v_\mfP(\mu) <0$, we get $v_\mfP(\mu +1)=v_\mfP(\mu)$ and hence  $v_\mfP(j_{E^\prime})\geq 0$. Next for $v_\mfP(\mu)=0$, we obtain $v_\mfP(\mu +1)\geq 0$ and hence $v_\mfP(j_{E^\prime}) \geq 0$. Finally, for $0 <v_\mfP(\mu) \leq 8 v_\mfP(2)$, we have $v_\mfP(\mu +1)=0$ and hence $v_\mfP(j_{E^\prime})=8v_\mfP(2)- v_\mfP(\mu) \geq 0$.
	%\begin{itemize}
	%	\item If $v_\mfP(\mu) <0$, then $v_\mfP(\mu +1)=v_\mfP(\mu)$. 
	%	By \eqref{inequality for valution of mu}, we get $v_\mfP(j_{E^\prime})\geq 0$.
	%	\item  If $v_\mfP(\mu)=0$, then $v_\mfP(\mu +1)\geq 0$, hence $v_\mfP(j_{E^\prime})\geq 8v_\mfP(2) \geq 0$.
	%	\item  If $v_\mfP(\mu) >0$, then $v_\mfP(\mu +1)=0$.
	%	By \eqref{inequality for valution of mu}, we get $v_\mfP(j_{E^\prime})=8v_\mfP(2)- v_\mfP(\mu) \geq 0$. 
	%\end{itemize}
	In all these cases, we get $v_\mfP(j_{E^\prime}) \geq 0$, which is a contradiction to Proposition~\ref{auxilary result for main result2}(3). This completes the proof of the theorem.
\end{proof}

\begin{proof}[Proof of Corollary~\ref{cor to main result2}]
	Suppose $(a,b,c) \in \mcO_K^3$ is a non-trivial primitive solution to $x^5+y^5=dz^p$ of exponent $p>V$ with $c\neq \pm1$ and $2 \nmid c$, where $V$ is  as chosen in Theorem~\ref{main result2 for (r,r,p)}. As the prime $\mfP$ is the unique prime  of $\mcO_K$ lying above $2$, we get that $\mfP \nmid c$ and hence $(a,b,c) \in W_K$, which is contraction to Theorem~\ref{main result2 for (r,r,p)}. This completes the first part of the corollary. The second part of the corollary follows from Corollary~\ref{cor to main result1} with $r=5$.
\end{proof}

\section{Local criteria for Theorem~\ref{main result1 for (r,r,p)}}
\label{section for local criteria}
%\subsection{Local criteria for Theorem~\ref{main result1 for (r,r,p)}}
In this section, we provide various local criteria on $K$ such that Theorem~\ref{main result1 for (r,r,p)} holds over $K^+$.
\subsection{$d$ is a rational prime}
\label{section for d is a rational prime}
%In this section, we will provide various local criteria of $K$ satisfying Theorem~\ref{main result1 for (r,r,p)}. 
In this subsection, we assume $d>2$ is a rational prime. Let $h_F^+$ denote the narrow class number of a number field $F$.
%and give various local criteria on $K$ such that Theorem~\ref{main result1 for (r,r,p)} holds over $K^+$.
%\begin{dfn}
%	Let $h_K^+$ denote the narrow class number of $K$.
%\end{dfn}

\begin{prop}
	\label{loc crit for d rational prime}
	Let $K$ be a totally real number field and let $r \geq 5$ be a rational prime. Let $K^+:=K(\zeta_r+ \zeta_r^{-1})$ and $\pi :=\zeta_r + \zeta_r^{-1}-2$. Assume that $d>2$ is a rational prime and 
	\begin{enumerate}
		\item $2 \nmid h_{K^+}^+$ and $2$ is inert in $K^+$ (let  $\mfP:=2\mcO_{K^+}$);
		\item $d \equiv 1 \pmod {\mfP^2}$, $d$ is inert in $K^+$ and $r$ is inert in $K$;
		%		\item Assume $2$ is either inert or totally ramified in $K^+$ with unique prime $\mfP|2$ of ramification index $e:= e(\mfP,2)$;
		\item the congruences $\pi \equiv \nu^2 \pmod {\mfP^5}$,  $d \equiv v^2 \pmod {\mfP^5}$ and $\pi d \equiv \delta^2 \pmod {\mfP^5}$ have no solutions for $ \nu, v, \delta \in \mcO_{K^+}/ {\mfP^5}$;
	\end{enumerate}
	Then the conclusion of Theorem~\ref{main result1 for (r,r,p)} holds over $K^+$.
\end{prop}
To prove the above proposition, we need the following lemma.
\begin{lem}
	\label{integral}
	Let $(\lambda, \mu)$ be a solution to the $S_{K^+,2rd}$-unit equation
	$\lambda+\mu=1, \ \lambda, \mu \in \mcO_{S_{K^+,2rd}}^\ast$, and let $\mfP \in S_{K^+, 2}$. Then there exists $\lambda', \mu' \in \mcO_{S_{K^+,2rd}}^\ast$ with $v_\mfP(\lambda') \geq 0$ and $v_\mfP(\mu') \geq 0$ such that $\lambda'+\mu'=1$ and $	\max \left\{|v_\mfP(\lambda)|,|v_\mfP(\mu)| \right\}=	\max \left\{|v_\mfP(\lambda')|,|v_\mfP(\mu')| \right\}$.
\end{lem}
\begin{proof}
	Let $\lambda, \mu \in \mcO_{S_{K^+,2rd}}^\ast$ with $\lambda+\mu=1$. If  $v_\mfP(\lambda) \geq 0$, then $v_\mfP(\mu) \geq 0$ since $\lambda+\mu=1$. Choose $\lambda'=\lambda$ and $\mu'=\mu$, hence we are done. If  $v_\mfP(\lambda)<0$, then $v_\mfP(\mu)=v_\mfP(\lambda) < 0$ since $\lambda+\mu=1$. Choose $\lambda'= \frac{1}{\lambda}$ and $\mu'= \frac{\mu}{\mu-1}$. Clearly, $v_\mfP(\lambda') \geq 0$ and $v_\mfP(\mu') \geq 0$. In all these cases, we have $\lambda'+\mu'=1$  and $	\max \left\{|v_\mfP(\lambda)|,|v_\mfP(\mu)| \right\}=	\max \left\{|v_\mfP(\lambda')|,|v_\mfP(\mu')| \right\}$. 
\end{proof}
We now prove Proposition~\ref{loc crit for d rational prime}.
\begin{proof}[Proof of Proposition~\ref{loc crit for d rational prime}]
	To prove this proposition, it is enough to prove that every solution $(\lambda, \mu)$ to the $S_{K^+,2rd}$-unit equation
	$\lambda+\mu=1, \ \lambda, \mu \in \mcO_{S_{K^+,2rd}}^\ast$ satisfies
	$	\max \left\{|v_\mfP(\lambda)|,|v_\mfP(\mu)| \right\}\leq 4v_\mfP(2)=4.$
	Suppose there exists a solution $(\lambda, \mu)$ to the $S_{K^+,2rd}$-unit equation $\lambda+\mu=1$ such that $	\max \left\{|v_\mfP(\lambda)|,|v_\mfP(\mu)| \right\} \geq 5.$ Then by Lemma~\ref{integral}, there exist  $\lambda', \mu' \in \mcO_{S_{K^+,2rd}}^\ast$ with $\lambda'+\mu'=1$, $v_\mfP(\lambda') \geq 0$ and $v_\mfP(\mu') \geq 0$ such that $\max \left\{v_\mfP(\lambda'),v_\mfP(\mu') \right\} \geq 5$. Without loss of generality, take $v_\mfP(\mu') \geq 5$.
	Since the $S_{K^+,2rd}$-unit equation $\lambda'+\mu'=1$ has only finitely many solutions, choose $(\lambda', \mu')$ such that $v_\mfP(\mu')$ as large as possible.
	Since  $v_\mfP(\mu') \geq 5$ and $\lambda'+\mu'=1$,  we have $v_\mfP(\lambda')=0$. Hence $\lambda' \in \mcO_{S_{K^+, rd}}^\ast$ and $\lambda' \equiv 1 \pmod {\mfP^5}$. Now, we will show that $\lambda'$ is a square in $\mcO_{S_{K^+, rd}}^\ast$.
	
	Since  $\lambda' \in \mcO_{S_{K^+, rd}}^\ast$, we have 
	\begin{equation}
		\label{factor of lambda}
		\lambda'= \alpha \pi^sd^t,
	\end{equation}
	where $\alpha \in \mcO_{K^+}^\ast$ and $s,t \in \Z$. To show $\lambda'$ is a square, it is enough to show that $\alpha$ is a square and $s,t$ are even. First, we will show that $\alpha$ is a square.
	Since $\left(\zeta_r^{\frac{r-1}{2}} + \zeta_r^{-\frac{r-1}{2}} \right)^2-4=\pi$, we have $\pi \equiv \epsilon^2 \pmod {\mfP^2}$, where $\epsilon:=\zeta_r^{\frac{r-1}{2}} + \zeta_r^{-\frac{r-1}{2}} \in \mcO_{K^+}^\ast$. Since $\lambda' \equiv 1 \pmod {\mfP^5}$ and $d \equiv 1 \pmod {\mfP^2}$, by equation~\eqref{factor of lambda} we have $\alpha \equiv \beta^2 \pmod {\mfP^2}$, where $\beta:= \epsilon^{-s} \in \mcO_{K^+}^\ast$. Hence $\frac{\beta^2-\alpha}{4}\in \mcO_{K^+}$. 
	
	If possible, assume that $\alpha$ is not a square in $\mcO_{K^+}$. Consider the field $L=K^+(\sqrt{\alpha})= K^+(\frac{\beta -\sqrt{\alpha}} {2})$. Then $[L:K^+]=2$. The minimal polynomial of $\eta:=\frac{\beta -\sqrt{\alpha}} {2}$ over $K^+$ is $m_\eta(X)= X^2-\beta X+  \frac{\beta^2 - \alpha}{4} \in \mcO_{K^+}[X]$ with its discriminant $\alpha \in  \mcO_{K^+}^\ast$. 
	%		 Since $v_\mfq(\beta)=0$ for all $\mfq \in S_K$ and $S_K=S_K^\prime$, $\beta$ is a unit in $K$. 
	Therefore, $L$ is an unramified extension of degree 2 over $K^+$, which contradicts the hypothesis $2\nmid h_{K^+}^+$. Hence $\alpha$ is a square.
	
	Next, we will show that both $s$ and $t$ are even. First, assume $s$ is odd and $t$ is even. Let $s=2k+1$. Since $\lambda' \equiv 1 \pmod {\mfP^5}$, by~\eqref{factor of lambda} we have $\pi \equiv \nu^2 \pmod {\mfP^5}$ for some $\nu \in \mcO_{K^+}$, which contradicts the hypothesis $(3)$. In the other cases, proceeding in a similar way, we get a contradiction to the hypothesis $(3)$. Hence both $s$ and $t$ are even. Thus $\lambda'$ is a square in  $ \mcO_{S_{K^+, rd}}^\ast$.
	
	Now, the proof follows from \cite[Theorem 6]{M23}  and we include it for completeness.
	Let $\lambda'=\gamma^2$, for some $\gamma \in \mcO_{S_{K^+, rd}}^\ast$. This gives $\mu'= 1-\gamma^2= (1+\gamma)(1-\gamma)$. Let $t_0:= v_\mfP(\mu')$, $t_1:= v_\mfP((1+\gamma))$ and $t_2:= v_\mfP((1-\gamma))$. This gives $t_0= t_1+t_2$. Since $t_0 \geq 5$, we have either $t_1=1$ or $t_2=1$ but both can't be equal to $1$. Without loss of generality, take $t_1=1$. This gives $t_2= t_0-1$.
	Choose $\lambda''= \frac{(1+\gamma)^2}{4\gamma}$ and $\mu''= \frac{-(1-\gamma)^2}{4\gamma}$. Then $\lambda'' + \mu ''=1$. So $(\lambda'', \mu'')$ is a solution to the $S_{K^+,2rd}$-unit equation $\lambda+\mu=1$. Now $v_\mfP(\mu'')=2  t_2-2=2  t_0-4>t_0= v_\mfP(\mu')$, which is a contradiction since $v_\mfP(\mu')$ is the largest among all the solutions to the $S_{K^+,2rd}$-unit equation.
	This completes the proof of the proposition.
\end{proof}

In the following corollary, we give the local criteria when $K =\Q$.
\begin{cor}
	\label{cor1 to loc crit d prime}
	Let $r\geq 7$ be a rational prime with $r \not\equiv 1 \pmod 8$ and $\Q^+:=\Q(\zeta_r + \zeta_r^{-1})$. Suppose $d>2$ is a rational prime and 
	\begin{enumerate}
		\item $2 \nmid h_{\Q^+}^+$;
		\item $(2)$ is inert in $\Q^+$;
		\item $d \equiv 1 \pmod 4$, $d$ is inert in $K^+$;
		\item $d^{\frac{r-1}{2}}, rd^{\frac{r-1}{2}} \not\equiv 1,9,17,25 \pmod {32}$;
	\end{enumerate}
	Then the conclusion of Theorem~\ref{main result1 for (r,r,p)} holds over $\Q^+$.
\end{cor}
\begin{remark}
	Note that the above corollary is not true for $r=5$. This follows from the fact that for all primes $d >2$, we have $d^{\frac{r-1}{2}} =d^2 \pmod {32} \in  \{1,9,17,25\}$.
\end{remark}
\begin{example}
	Corollary~\ref{cor1 to loc crit d prime} holds for $r=7$ with $d =5,37, 53$ and for $r=11$ with $d=5,13, 29, 37, 53$.
\end{example}
%As an immediate corollary to Corollary~\ref{cor1 to loc crit d prime}, we have the following result.
%\begin{cor}
\label{cor2 to loc crit d prime}
%	Let $r=7$, $d \geq 5$ be any rational prime satisfying $d \equiv 5 \pmod {32}$
%Let $r=7$ with $d =5,37, 101, 197, 229$ and $r=11$ with $d=5,13, 29$. Let $\Q^+:=\Q(\zeta_7 + \zeta_7^{-1})$. 
%Then the conclusion of Theorem~\ref{main result1 for (r,r,p)} holds over $\Q^+$. 
%In particular, equation $x^7+y^7=dz^p$ has no asymptotic solution $(a,b,c) \in \Z^3$ with $2|c$.
%\end{cor}

To prove Corollary~\ref{cor1 to loc crit d prime}, we need to recall the following basic result in algebraic number theory, which can be conveniently found in \cite[Lemma 31]{M23}.
\begin{lem} 
	\label{Norm red lemma}
	Let $\lambda, v \in \mcO_{K^+}$ and let $n \in \N$. Let $\mfP$ be the unique prime of $K^+$ lying above $2$. If $\lambda \equiv v^2 \pmod {\mfP^n}$, then $\text{Norm}_{K^+/K}(\lambda) \equiv \text{Norm}_{K^+/K}(v^2) \pmod {\mfQ^{[n/{e'}]}}$, where  $\mfQ$ be the unique prime of $K$ lying above $2$ and $e':=e(\mfP/\mfQ)$.
\end{lem}
\begin{proof}[Proof of Corollary~\ref{cor1 to loc crit d prime}]
	To prove Corollary~\ref{cor1 to loc crit d prime}, we need to check that all the hypotheses of Proposition~\ref{loc crit for d rational prime} are satisfied in the settings of Corollary~\ref{cor1 to loc crit d prime}.
	Here $K= \Q$, clearly the condition (1) of Proposition~\ref{loc crit for d rational prime} is satisfied. Since $d \equiv 1 \pmod4$, it follows that $d \equiv 1 \pmod{\mfP^2}$ and hence (2) is satisfied. We will prove the condition (3) by contradiction.
	
	Suppose there exists $\nu \in \mcO_{K^+}$ such that $\pi \equiv \nu^2 \pmod {\mfP^5}$. Applying Lemma~\ref{Norm red lemma} to $K =\Q$, $\lambda =\pi$, $n=5$, we get $r \equiv \alpha^2 \pmod {2^5}$, for some $\alpha \in \Z$. We arrive at a contradiction to $r \not\equiv 1 \pmod 8$ since only odd squares modulo $32$ are $\{1,9, 17, 25\}$. 
	
	Suppose there exists $v \in \mcO_{K^+}$ such that $d \equiv v^2 \pmod {\mfP^5}$. Applying Lemma~\ref{Norm red lemma} to $K =\Q$, $\lambda =d$, $n=5$, we get $d^{\frac{r-1}{2}} \equiv \beta^2 \pmod {2^5}$, for some $\beta \in \Z$. We get a contradiction, as $d^{\frac{r-1}{2}} \not\equiv 1,9,17,25 \pmod {32}$ by our assumption. 
	
	Suppose there exists $\delta \in \mcO_{K^+}$ such that $\pi d \equiv \delta^2 \pmod {\mfP^5}$. Applying Lemma~\ref{Norm red lemma} to $K =\Q$, $\lambda = \pi d$, $n=5$, we get $rd^{\frac{r-1}{2}} \equiv \gamma^2 \pmod {2^5}$, for some $\gamma \in \Z$. We get a contradiction, as $ rd^{\frac{r-1}{2}} \not\equiv 1,9,17,25 \pmod {32}$. 
	%Since only odd squares modulo $32$ are $\{1,9, 17, 25\}$, we get a contradiction to $ rd^{\frac{r-1}{2}} \not\equiv 1,9,17,25 \pmod {32}$. 
	Hence the condition (3) is satisfied. We are done with the proof.
\end{proof}

\subsection{Let $d\in \mcO_K^\ast$}
\label{section for d is a unit}
We now consider the case where $d \in \mcO_K^\ast$. The results in this subsection are analogous to the corresponding results of \cite{M23} on $x^r+y^r=z^p$.
%to $x^r+y^r=dz^p$ for $d \in \mcO_K^\ast$.
\begin{prop}
	Let $K$ be a totally real number field and let $r \geq 5$ be a rational prime. Let $K^+=K(\zeta_r+ \zeta_r^{-1})$ and $\pi :=\zeta_r + \zeta_r^{-1}-2$. Assume that $d \in \mcO_K^\ast$ and
	\begin{enumerate}
		\item $2 \nmid h_{K^+}^+$ and $r$ is inert in $K$;
		\item Assume that there exists a unique prime $\mfP$ of $K^+$ lying above $2$ with ramification index $e:= e(\mfP,2)$;
		\item the congruence $\pi \equiv \nu^2 \pmod {\mfP^{4e+1}}$ has no solutions for $ \nu \in \mcO_{K^+}/ {\mfP^{4e+1}}$;
	\end{enumerate}
	Then the conclusion of Theorem~\ref{main result1 for (r,r,p)} holds over $K^+$.
\end{prop}

\begin{proof}
	Recall that $S_{K^+, 2rd}=\{\mfp: \mfp\in P_{K^+} \text{ with } \mfp |2rd\}.$ Since $d\in \mcO_K^\ast$, we have $ S_{K^+, 2rd}=S_{K^+,2r} =S_{K^+, 2} \cup S_{K^+, r}$.  Now, the proof of \cite[Theorem 6]{M23} shows that the conditions of Theorem~\ref{main result1 for (r,r,p)} are satisfied, hence the conclusion of Theorem~\ref{main result1 for (r,r,p)} holds over $K^+$.
	%	Hence, the proof of the proposition follows from Theorem~\ref{main result1 for (r,r,p)} and \cite[Theorem 6]{M23}.
	%since $S_{K^+, 2rd}= S_{K^+, 2} \cup S_{K^+, r}$.
\end{proof}
The following two corollaries are generalizations of \cite[Corollaries 7-9]{M23} from $x^r+y^r=z^p$ to $x^r+y^r=dz^p$. Corollary~\ref{loc crit for Q} (respectively Corollary~\ref{loc crit for quaqdratic field}) gives local criteria when $K =\Q$ (respectively $K$ is a real quadratic field), and its proof follows from Theorem~\ref{main result1 for (r,r,p)} since $S_{K^+,2rd}= S_{K^+, 2r}$.
\begin{cor}
	\label{loc crit for Q}
	Let $r\geq 5$ be a rational prime with $r \not\equiv 1 \pmod 8$ and $\Q^+:=\Q(\zeta_r + \zeta_r^{-1})$. Assume that $d \in \mcO_K^\ast$. Suppose one of the following conditions holds.
	\begin{enumerate}
		\item $2$ is inert in $\Q^+$ and $2 \nmid h_{\Q^+}^+$;
		\item $r \in \{5,7,11,13, 19, 23, 37, 47, 53, 59, 61, 67, 71, 79, 83, 101, 103, 107, 131, 139, 149\}$. 
	\end{enumerate}
	Then the conclusion of Theorem~\ref{main result1 for (r,r,p)} holds over $\Q^+$. 
	%In particular, by Corollary~\ref{cor to main result1}, the equation $x^r+y^r=dz^p$ has no asymptotic solution $(a,b,c) \in \Z^3$ with $2 |c$.
\end{cor}
%The following corollary is an analog of \cite[Corollary 9]{M23} which gives local criteria when $K$ is a real quadratic field, and its proof follows from Theorem~\ref{main result1 for (r,r,p)}. %since  $S_{K^+,2rd}= S_{K^+, 2r}= S_{K^+, 2} \cup S_{K^+, r}$.
\begin{cor}
	\label{loc crit for quaqdratic field}
	Let $r\geq 5$ be a rational prime. For any square-free positive integer $t \geq 2$, let $K=\Q(\sqrt{t})$ and $K^+:=K(\zeta_r + \zeta_r^{-1})$. Assume that $d \in \mcO_K^\ast$ and
	\begin{enumerate}
		\item $r$ is inert in $K$ and $r \not \equiv 1, t\pmod 8$;
		\item $2 \nmid h^+_{K^+}$;
		\item there exists a unique prime $\mfP$ of $K^+$ lying above $2$;
		%Let $\mfQ \in P_K$ with $\mfQ |2$.  
	\end{enumerate}
	Then the conclusion of Theorem~\ref{main result1 for (r,r,p)} holds over $K^+$. 
	%In particular, by Corollary~\ref{cor to main result1}, there exists a constant $V=V_{K,r,d}>0$ (depending on $K,r,d$) such that for primes $p>V$, the equation $x^r+y^r=dz^p$ with $p>V$ has no non-trivial primitive solution $(a,b,c) \in \mcO_K^3$ with $\mfQ |c$.
\end{cor}
\begin{example}
	Corollary~\ref{loc crit for quaqdratic field} holds for $t=2$ with $r \in \{5,11,13 \}$ and $\mfP=(\sqrt{2})\mcO_{K^+}$, and for $t=5$ with $r=7$ and $\mfP=(2)\mcO_{K^+}$.
\end{example}

\section*{Acknowledgments} 
The authors thank Diana Mocanu for discussing the article \cite{M23} during the Lorentz Center Workshop: ``An Expedition into Arithmetic Geometry". We also thank Nuno Freitas for answering several questions and comments. We also thank Narasimha Kumar for the insightful discussions. The authors are grateful to the anonymous referee for a careful reading of the manuscript and for the mathematical suggestions and comments towards the improvement of this article. S. Jha is supported by ANRF grant CRG/2022/005923.
% at the Lorentz Center in Leiden.
%: An Expedition into Arithmetic Geometry at the Lorentz Center in Leiden, The Netherlands, from May 30 to June 2, 2023. SERB grant CRG/2022/005923 partially supported the second author's research.

%\begin{thebibliography}{abc9999}

\end{document}